\documentclass[11pt]{amsart}

\usepackage{amssymb}
\usepackage{amsmath}
\usepackage{mathrsfs}
\usepackage{setspace}
\usepackage{mathtools}
\usepackage{verbatim}
\usepackage{lipsum}
\usepackage{amsthm}
\usepackage{tikz-cd}
\usepackage{hyperref}
\usepackage[all]{xy}
\usepackage{enumitem}
\usepackage{xspace}

\newtheorem{theorem}{Theorem}[section]
\newtheorem{lemma}[theorem]{Lemma}
\newtheorem{corollary}[theorem]{Corollary}

\newtheorem{remark}[theorem]{Remark}

\newtheorem{proposition}[theorem]{Proposition}
\newtheorem{defn}[theorem]{Definition}

\newtheorem{problem}{Problem}
\numberwithin{equation}{section}

\newcommand{\bb}[1]{\mathbb{#1}}

\newcommand{\cal}[1]{\mathcal{#1}}
\newcommand{\ff}[1]{\mathfrak{#1}}

\newcommand{\mftq}{\mathcal F_{\mathcal X^{(4)}}}
\newcommand{\mftqe}{\mathcal F_{\mathcal X_{\bar{\eta}_1}^{(4)}}}
\newcommand{\mfttt}{\mathcal F_{\mathcal X^{(4)}_t}}
\newcommand{\mxttt}{\mathcal X^{(4)}_t}
\newcommand{\mxtq}{\mathcal X^{(4)}}
\newcommand{\mxtqe}{\mathcal X_{\bar{\eta}_1}^{(4)}}

\newcommand{\alc}{adjoint log canonical\xspace}
\newcommand{\acanonical}{adjoint canonical\xspace}
\newcommand{\ccurves}[1]{\overline{NE(#1)}}

\DeclareMathOperator{\bir}{Bir}
\DeclareMathOperator{\autom}{Aut}
\DeclareMathOperator{\volume}{vol}

\makeatletter
\@namedef{subjclassname@2020}{\textup{2020} Mathematics Subject Classification}
\makeatother

\begin{document}
\title[Effective generation for foliated surfaces]{Effective generation for foliated surfaces: results and applications}
\date{\today}

\author{Calum Spicer}
\address{Department of Mathematics, King's College London, London WC2R 2LS, United Kingdom.}
\email{calum.spicer@kcl.ac.uk}

\author{Roberto Svaldi}
\address{EPFL, SB MATH-GE, MA B1 497 (B\^{a}timent MA), Station 8, CH-1015 Lausanne, Switzerland.}
\email{roberto.svaldi@epfl.ch}

\begin{abstract}
We explore the birational structure and invariants of a foliated surface $(X, \cal F)$ in terms of the adjoint divisor $K_{\cal F}+\epsilon K_X$, $0< \epsilon \ll 1$.  
We then establish a bound on the automorphism group of an adjoint general type foliated surface $(X, \mathcal F)$,
provide a bound on the degree of a general curve invariant by an algebraically integrable foliation
on a surface and prove that the set of $\epsilon$-adjoint canonical models of foliations
of general type and with fixed volume form a bounded family.
\end{abstract}

\keywords{Foliations on surfaces, Minimal Model Program, Singularities, Effective birationality}
\subjclass[2020]{14E30, 37F75, 32S65}
\thanks{RS was partially supported from the European Union's Horizon 2020 research and innovation programme under the Marie Sk\l{}odowska-Curie grant agreement No. 842071.}

\maketitle

\setcounter{tocdepth}{1} 
\tableofcontents

\section{Introduction}

A central challenge in the study of the birational geometry of varieties is to understand the behavior 
of the pluricanonical maps
\begin{align*}
\phi_{\vert mK_X \vert}\colon X \dashrightarrow \bb PH^0(X, mK_X)
\end{align*}
as a function of $m \in \mathbb N$, for those varieties which admit non-trivial pluricanonical forms.  

In this paper, we are interested in studying this question for rank one foliations on surfaces $(X,\cal F)$.  
Already for surface foliations the problem of understanding the maps induced by the sections of 
$mK_\mathcal{F}$ appears to be quite challenging.
To remedy this issue, we instead consider a slight perturbation of this problem: namely, we aim to understand the behavior of the maps
\begin{align*}
\phi_{\vert mK_{\cal F}+nK_X\vert } \colon 
X \dashrightarrow 
\mathbb P H^0(X, mK_{\cal F}+nK_X), 
\qquad m \gg n>0.
\end{align*}
Along the way, we prove several new results on the birational structure of foliations in terms of the adjoint divisors $K_{\cal F}+\epsilon K_X$, $0< \epsilon \ll 1$.
Considering adjoint divisors of this form is a natural approach to the study of foliated varieties as it allows us
to apply classical results on the positivity of the canonical bundle of varieties that may not hold if one just considers 
the canonical bundle of the foliation, cf.~\cite{PS16}.

Among other applications, these new structural results provide bounds on the automorphism groups of foliations (Theorem~\ref{thm_auto_intro}),
bounds on the degree of curves invariant by an algebraically integrable foliation (Theorem~\ref{thm_degree_bound_intro}).
Finally, we also prove a boundedness results for generalizations of surface foliations of general type, that constitutes important progress towards the construction of a moduli space of this class of foliations (Theorem~\ref{thm_bdd_intro}).

\subsection{Adjoint MMP}

Our first main result is the proof of the existence and termination of the MMP for divisors of the form $K_{\cal F}+\epsilon K_X$, for $0<\epsilon \ll 1$.  
While the existence (and termination) of the MMP for $K_X$ is classical and the MMP for $K_{\cal F}$ is well known,~\cite{McQuillan08}, it is not a priori clear that the sum of these results automatically  implies that one can run a $(K_{\cal F}+\epsilon K_X)$-MMP.
Moreover, even assuming that it was possible to run such MMP, it is not a priori clear how to bound the singularities of the foliation and variety on the minimal model.

\begin{theorem}[= Theorem~\ref{thm_adj_mmp}]
\label{thm_intro_mmp}
Let $X$ be a smooth projective surface and $\cal F$ a rank one foliation with canonical singularities.
Then, for any $0<\epsilon<\frac{1}{5}$
there exists a birational morphism $\varphi\colon X \rightarrow Y$ such that either
\begin{enumerate}
\item $K_{\cal G}+\epsilon K_Y$ is nef, where $\cal G = \varphi_*\cal F$; or,

\item there exists a morphism $f\colon Y \rightarrow Z$ auch that $\rho(Y/Z) = 1$ and $-(K_{\cal G}+\epsilon K_Y)$ is $f$-ample.
\end{enumerate}

Moreover, $Y$ has klt singularities and $\cal G$ has log canonical singularities.
\end{theorem}

We also prove the existence of $\epsilon$-adjoint canonical models.

\begin{theorem}[= Corollary~\ref{cor_can_model}]
\label{thm_intro_can_model}
Notation as in Theorem \ref{thm_intro_mmp}.  Suppose in addition that $K_{\cal F}$ is big.
Then, there exists a birational morphism $p\colon (Y, \cal G) \rightarrow (Y_{can}, \cal G_{can})$
such that
\begin{enumerate}
\item $Y_{can}$ is projective;

\item $K_{\cal G_{can}}+\epsilon K_{Y_{can}}$ is an ample $\bb Q$-Cartier divisor; and 

\item $Y_{can}$ has klt singularities and $\cal G_{can}$ has log canonical singularities.
\end{enumerate}
\end{theorem}

This $\epsilon$-adjoint canonical model must be contrasted with McQuillan's notion of a canonical model
of a foliation where the underlying space is, a priori,  only an algebraic space, and its projectivity is not known.
Our notion of an $\epsilon$-\acanonical model should also be compared with the minimal partial du Val resolution of the canonical model of a surface foliation, see~\cite{Chen19}.

We are also able to provide a precise
statement on the singularities of the underlying variety 
which arise in this MMP, see Corollary~\ref{cor_bounded_sing}.  
This control on the singularities which arise in the course of the MMP is one of the key advantages of working with adjoint foliated divisors rather than simply with the canonical divisor of a foliation.

\subsection{Boundedness and effective birationality}

Our next main result is a boundedness result for $\epsilon$-\acanonical models
of foliations of general type.

\begin{theorem}[= Theorem \ref{thm_bdd_ample_models}]
\label{thm_bdd_intro}
There exists a universal real constant $\tau > 0$ such that the following statement holds:
\\
Fix positive real numbers $C, 0<\epsilon <\tau$.
The set of foliated pairs
\begin{align*}
\mathcal{M}_{2, \epsilon, C}:= &
\{
(X, \mathcal F) \ | \ X 
\text{ is a projective klt surface}, 
\ \cal F \text{ is rank one}, \\
& (X, \mathcal F) \text{ is an $\epsilon$-\acanonical foliated pair, $K_{\cal F}$ is big}, 
\\ &
K_{\cal F}+\epsilon K_X \text{ is ample, and } (K_{\cal F}+\epsilon K_X)^2 \leq C
\}
\end{align*}
forms a bounded family.
\end{theorem}

We refer to Section \ref{s_ed_alc_defns} for the precise definition of $\epsilon$-\acanonical, 
but we remark here that it is a natural assumption on the singularities of $(X, \mathcal F)$.
Analogous results for a more restricted class of foliated pairs have recently appeared also in~\cite{Chen21}.

The key technical ingredient in the above statement is the following
effective birationality statement
which follows from our results on the MMP 
and some new results of Birkar on adjoint linear series,~\cite{Bir20}.

\begin{theorem}[= Corollary~\ref{cor_main_corollary}]
\label{thm_intro_bnds}
Let $\tau > 0$ be the constant whose existence is established in Theorem~\ref{thm_bdd_intro}.
Then for all $0<\epsilon<\tau$ there exists a positive integer 
$M = M(\epsilon)$
such that the following statement holds:
\\
Let $X$ be a smooth projective surface and let $\cal F$ be a rank one foliation on $X$ with canonical singularities.
Suppose that $K_{\mathcal F}$ is big.  
Then 

\begin{enumerate}
\item $K_{\mathcal F}+\epsilon K_X$ is big; and

\item $|M(K_{\cal F}+\epsilon K_X)|$ defines a birational map.
\end{enumerate}
\end{theorem}

In fact, we are able to prove versions of Theorems~\ref{thm_intro_mmp},~\ref{thm_intro_can_model},~\ref{thm_bdd_intro} 
and~\ref{thm_intro_bnds} which allow for the presence of a boundary divisor.
 
Theorem \ref{thm_intro_bnds} also supplies a partial answer to~\cite[Problem 6.8]{PS16}.
To provide a complete answer to this problem would require an exact value on the universal constant 
$\tau$ in Theorem~\ref{thm_intro_bnds}.

\begin{problem}
Determine an effective upper bound for $\tau$.
\end{problem}

We also remark that as examples in \cite{Lu21} show that there does not exist a universal $M$
such that $|MK_{\cal F}|$ defines a birational map, so to get an effective birationality statement the small perturbation by $K_X$
is necessary.

\subsection{Numerical invariants of surface foliations}

We are able to provide several applications of the above results to the study of numerical invariants of foliated surfaces, automorphism groups of foliations
and to the study of curves invariant by foliations.

Given a big divisor $\bb Q$-Cartier divisor $D$ the volume $\mathrm{vol}(D)$ is defined to be 
\begin{align*}
\mathrm{vol}(D) :=
\limsup_{m \to \infty} \frac{h^0(mD)}{m^n/n!}.
\end{align*}
By \cite[Corollary 2.2.45]{Lazarsfeld04a} we may uniquely extend the volume to a function on $\mathbb R$-Cartier divisors.
The volume is a fundamental invariant in birational geometry and, in analogy with the classical MMP, cf.~\cite{HMX12},
we expect the set 
\begin{align*}
\mathcal V_n := 
&\{
\mathrm{vol}(X, K_{\cal F}) \ \vert \ 
\cal F 
\text{ is a rank one foliation of general type with canonical} \\ 
& 
\text{ singularities and $X$ is a klt projective variety of dimension $n$}
\}
\end{align*}
to be highly structured for each fixed dimension $n$.
In particular, we expect $\mathcal V_n$ to be bounded away from $0$.  
This is a challenging problem, already for $n=2$, but by perturbing $K_{\cal F}$ slightly we can verify 
a related prediction.

\begin{theorem}[= Theorem \ref{thm_vol_bound}]
Let $\tau > 0$ be the constant  whose existence is established in Theorem~\ref{thm_bdd_intro}. 
Then, for all $0<\epsilon<\tau$ there exists $0< v(\epsilon)$ such that the following statement holds:
\\
If $X$ is a smooth projective surface, 
$\cal F$ is a rank one foliation with canonical singularities and $K_{\cal F}$ is big,
then 
\begin{align*}
\mathrm{vol}(K_{\cal F}+\epsilon K_X) \geq v(\epsilon).
\end{align*}
\end{theorem}

In fact, we are able to prove the above statement allowing for a boundary divisor.
As a direct consequence of the above volume bound we get another bound on the automorphism
group of a foliation of general type.

\begin{theorem}[= Theorem \ref{thm_automorphism_bound}]
\label{thm_auto_intro}
Let $\tau > 0$ be the constant whose existence is established in Theorem~\ref{thm_bdd_intro}.
Then for all $0<\epsilon<\tau$ there exists $0<C = C(\epsilon)$ such that the following statement holds:
\\
If $X$ is a smooth projective surface, 
$\cal F$ is a rank one foliation with canonical singularities and $K_{\cal F}$ is big, then
\begin{align*}
\#\text{Bir}(X, \cal F) \leq C\cdot \mathrm{vol}(K_{\cal F}+\epsilon K_X).
\end{align*}
\end{theorem}

In analogy with the classical situation it would be nice to find a bound which depends only
on $\mathrm{vol}(K_{\cal F})$.   See also \cite{CF14} for similar results in this direction.

We were also able to provide a refinement of the bound initially proven in \cite{PS16}, cf. also~\cite{HL, Chen19}.

\begin{theorem}[= Theorem \ref{thm_degree_bound}]
\label{thm_degree_bound_intro}
Let $\tau>0$ be the constant whose existence is established in Theorem~\ref{thm_bdd_intro}.  
Then for all rational numbers $0<\epsilon<\tau$ there exists $0<C = C(\epsilon)$ the following statement holds:
\\
Let $X$ be a smooth projective surface and let $\cal F$ be a a rank one foliation on $X$.
Assume that 
\begin{enumerate}
\item 
$K_{\mathcal F}$ is big,
\item 
$(X, \mathcal F)$ is $\epsilon$-\acanonical,
\item 
$\cal F$ admits a meromorphic first integral, and 
\item 
the closure of a general leaf, $L$, has geometric genus $g$.
\end{enumerate}
Then for any nef divisor $H$, 
\begin{align*}
H\cdot L \leq gC ~H\cdot (K_{\cal F}+\epsilon K_X).
\end{align*}
\end{theorem}

\subsection*{Acknowledgements}
We would like to thank Fabio Bernasconi, Paolo Cascini, Yen-An Chen, Christopher Hacon and Jorge V. Pereira for many valuable conversations, suggestions and comments.
We also thank the anonymous referee for many suggestions and improvements to the exposition of the paper.

\section{Preliminaries}
Throughout we work over an algebraically closed field $k$ of characteristic $0$.

We refer to \cite{Brunella00} for basic notions regarding foliations, and we refer to \cite{KM98} for basic
notions regarding the minimal model program.
We will assume throughout that all of our foliations are of rank one.

\subsection{ACC/DCC sets}

Given a subset $I \subset \mathbb R$ we say that $I$ satisfies the {\it ascending chain condition} (resp. {\it descending chain condition}), in short, ACC (resp. DCC), provided any increasing (resp. decreasing) sequence $x_n \in I$ is eventually constant.

Given a subset $I \subset [0, 1]$ and a $\mathbb R$-Weil divisor $\Delta$ on a normal variety, we write $\Delta \in I$ to indicate that all the coefficients of $\Delta$ are in $I$.

We will denote with $\cal S$ the following subset of $\mathbb R$, $\cal S := \{\frac{n-1}{n} \ \vert  n \in \mathbb N_{>0}\} \cup \{1\}$.

\subsection{Pairs and triples}
Let $X$ be a normal variety and let $\cal F$ be a foliation on
$X$.  Let $D$ be a $\bb R$-divisor on $X$.
We may uniquely decompose $D = D_{\text{inv}}+D_{\text{n-inv}}$ where the support of $D_{\text{inv}}$ is $\cal F$-invariant and no component of the support of $D_{\text{n-inv}}$ is $\cal F$-invariant.

By a {\bf (log) pair} $(X, \Delta)$ we mean the datum of a variety $X$ and an effective $\bb R$-divisor $\Delta$
such that $K_X+\Delta$ is $\mathbb R$-Cartier.

By a {\bf foliated (log) pair} $(\cal F, \Delta)$ on a variety $X$ we mean the datum of a foliation $\cal F$ on $X$
and an effective $\bb R$-divisor $\Delta$ such that $K_{\cal F}+\Delta$ is $\bb R$-Cartier.
When we assume that $X$ is projective, we shall say that the foliated pair is projective.

By a {\bf foliated triple} $(X, \cal F, \Delta)$ we mean the datum of a variety $X$, a foliation $\cal F$ on $X$
and an effective $\bb R$-divisor $\Delta$ such that both $K_X+\Delta$ and $K_{\cal F}+\Delta_{\text{n-inv}}$ are $\bb R$-Cartier.
If $\Delta= 0$ then we will just write $(X, \cal F)$ in place of $(X, \cal F, \Delta)$.
When we assume that $X$ is projective, we shall say that the foliated triple is projective.

\subsection{Transform of a foliation under a rational map}
\label{subs_transforms}

Let $X$ be a normal variety and let $\cal F$ be a foliation on $X$ and 
let $\phi\colon X' \dashrightarrow X$ be a dominant rational map.
Following \cite[\S 3.2]{Druel18} we may define the {\bf pulled back} foliation, denoted $\phi^{-1}\cal F$, on $X'$.
In the case where $\phi$ is birational, and $\mathcal G$ is a foliation on $X'$
we will denote by $\phi_*\mathcal G$ the pullback of $\mathcal G$
along the birational map $\phi^{-1}$ and refer to it a the {\bf transform} of $\mathcal G$ by $\phi$.

\subsection{Foliation singularities}
We are typically interested only in the case when $\Delta \geq 0$,
although it simplifies some computations to allow $\Delta$ to have negative coefficients.

Given a birational morphism $\pi\colon  \widetilde{X} \rightarrow X$ 
and a foliated pair $(\cal F, \Delta)$ on $X$,
let $\widetilde{\cal F}$ be the pulled back foliation on $\tilde{X}$.
We may write
\begin{align*}
K_{\widetilde{\cal F}}=
\pi^*(K_{\cal F}+\Delta)+ \sum a(E, \cal F, \Delta)E,
\end{align*}
where the sum runs over all the prime divisors of $\tilde X$ and 
\begin{align*}
\pi_* \sum a(E,\cal F,\Delta)E=-\Delta.
\end{align*}

The rational number $a(E,\cal F,\Delta)$ denotes the {\bf discrepancy} of $(\cal F,\Delta)$ with respect to $E$. 
\begin{defn}\label{d_canonical} Let $X$ be a normal variety and let $(\cal F,\Delta)$ be a foliated pair on $X$. 
We say that  $(\cal F, \Delta)$ is {\bf terminal} (resp. {\bf canonical},  {\bf klt}, {\bf log canonical}) if
$a(E, \cal F, \Delta) >0$ (resp. $\geq 0$, $> -\iota(E)$ and $\lfloor \Delta \rfloor =0$, $\geq -\iota(E)$),  for any birational morphism  $\pi\colon \tilde X\to X$ and for any prime divisor $E$ on  $\tilde X$, where
\begin{align*}
\iota(E) := 
\begin{cases}
1 \quad &\text{if $E$ is not $\mathcal F$-invariant}, \\
0 \quad &\text{if $E$ is $\mathcal F$-invariant}. \\
\end{cases}
\end{align*}

Moreover, we say that the foliated pair $(\cal F,\Delta)$ is {\bf log terminal} if $a(E,\cal F,\Delta)> -\iota(E)$
for any birational morphism  $\pi\colon \tilde X\to X$ and for any $\pi$-exceptional prime divisor $E$ on  $\tilde X$.
\end{defn}

We shall say that a foliated pair $(\cal F, \Delta)$ on a normal variety $X$ is {\bf strictly log canonical} at a point $x\in X$, if there exists a geometric valuation $E$ centered at $x$ such that $\iota(E)=1$ and $a(\cal F, \Delta, E)=-1$.
In particular, a strictly log canonical foliated pair is not canonical.

\begin{remark} 
Elsewhere in the literature $\iota(D)$
is denoted by $\epsilon(D)$.  However, in this paper we will frequently use
$\epsilon$ to denote a small positive real number and so to avoid confusion we have 
adopted this new notation.
\end{remark}

\begin{remark}
\begin{itemize}
	\item
The quantities $\iota(E)$ and $a(E, \cal F, \Delta)$
are independent of $\pi$.  If $\Delta = 0$ we will write $a(E, \cal F)$ for $a(E, \cal F, \Delta)$.
	\item
In the case where $\cal F = T_X$, no exceptional divisor
is invariant, i.e., $\iota(E)=1$, and so this definition recovers the usual
definitions of (log) terminal, (log) canonical, see \cite{KM98}.
In this case, we will write $a(E, X, \Delta)$ for $a(E, T_X, \Delta)$.
\end{itemize}
\end{remark}

\begin{defn}
Given a pair $(X, \Delta)$ and $\eta \geq 0$ we say
that $(X, \Delta)$ has $\eta$-lc singularities provided for all 
birational morphisms $\pi\colon X' \rightarrow X$ and $\pi$-exceptional divisors $E$ we have
$a(E, X, \Delta) \geq -(1-\eta)$.
\end{defn}

We say a foliated triple  $(X, \cal F, D)$ where $X$ is a surface and $\mathcal F$ has rank one is {\bf foliated log smooth}
provided $(X, D)$ has simple normal crossings and $\cal F$ has reduced singularities and each component of $D$ 
which is not invariant is everwhere transverse to
$\cal F$.  By \cite{Seidenberg} it is known that every surface foliated triple $(X, \cal F, D)$ admits
a resolution $\pi\colon X' \rightarrow X$ such that $(X', \pi^{-1}\cal F, \pi_*^{-1}D+E)$ is foliated log smooth
where $E = \text{exc}(\pi)$.  We call such a resolution a {\bf foliated log resolution}.

For a choice of $\epsilon > 0$, we define the $\epsilon$-adjoint log canonical divisor of a triple $(X, \cal F, \Delta)$ to be 
\begin{align*}
K_{(X, \cal F, \Delta), \epsilon}\coloneqq 
(K_{\cal F}+\Delta_{n-inv})+\epsilon (K_X+\Delta).
\end{align*}
We say a triple $(X, \cal F, \Delta)$ is adjoint general type (pseudo-effective, etc.) if for all $0<\epsilon\ll 1$ sufficiently small we have $K_{(X, \cal F, \Delta), \epsilon}$ is big (pseudo-effective, etc.).

Let $P \in X$ be a germ of a normal variety
and let $\ff m$ be the maximal ideal of $P$.  Let $\partial$ be a vector field
on $X$ which leaves $P$ invariant.  Since $\partial(\ff m) \subset \ff m$
we get an induced linear map $\partial_0\colon \ff m/\ff m^2 \rightarrow \ff m/\ff m^2$,
which we call the linear part of $\partial$ at $P$.

We recall the following characterization found in \cite[Fact I.1.8]{McQuillan08}.

\begin{proposition}
\label{prop_non_nilp_vf}
Let $\cal F$ be a germ of a rank one foliation on a normal variety $P \in X$ and suppose that $K_{\mathcal F}$ is Cartier
and that $P$ is $\mathcal F$-invariant.  
Let $\partial$ be a vector field generating $T_{\cal F}$ near $P$.
Then $\cal F$ is log canonical at $P$ if and only if $\partial_0$ is non-nilpotent.
\end{proposition}

\subsection{Basic definitions of the MMP}

We recall some of the main definitions commonly used in the Minimal Model Program. 
Let $X$ be a normal projective variety. 
We denote by $N_1(X)$ the $\mathbb R$-vector space spanned by $1$-cycles on $X$ modulo numerical equivalence (e.g. see \cite[Definition 1.16]{KM98}). 
We denote by $NE(X)\subset N_1(X)$ the subset of effective $1$-cycles $[\sum_{i=1}^k a_i C_i]$ where $a_1,\dots,a_k$ are positive real numbers and $C_1,\dots,C_k$ are curves in $X$, and we denote by $\overline {NE(X)}$ its closure (e.g. see \cite[Definition 1.17]{KM98}).
A {\bf ray} is a $1$-dimensional subcone $R$ of $\overline {NE(X)}$ and it is called {\bf extremal} if for any $u,v\in \overline {NE(X)}$ such that $u+v\in R$, we have that $u,v\in R$. 
If $D$ is a $\mathbb Q$-Cartier $\mathbb Q$-divisor on $X$ then the extremal ray $R$ is said to be $D$-{\bf negative} if $D\cdot C<0$ for any curve $C$ such that $[C]\in R$. 
A projective birational morphism $f\colon X \rightarrow Y$ between normal projective varieties is said to be an
{\bf extremal contraction} if the relative Picard number $\rho(X/Y)$ is equal to one. The extremal contraction is called a {\bf divisorial contraction} if its exceptional locus is a divisor. Given an extremal ray $R \subset \overline{NE}(X)$, an extremal contraction $f\colon X\to Y$ is said to be {\bf associated} to $R$ if the locus of $R$ coincides with the exceptional locus of $f$.

\subsection{Recollection on the foliated MMP}

We summarize some basic results on the existence of the MMP for surface foliations, 
as well as extending some well known results to the case of pairs $(\mathcal F, \Delta)$ with log canonical singularities.

\begin{lemma}
\label{lem_uniq_lc_place}
Let $P \in X$ be a germ of a surface singularity and let $\cal F$ be a rank one foliation on $X$ which is strictly log canonical at $P$. 
Let $\mu\colon Y \rightarrow X$ be any foliated log resolution which is an isomorphism over $X \setminus P$.  
Then there is exactly one $\mu$-exceptional divisor which is transverse to $\mu^{-1}\cal F$.  
\end{lemma}
\begin{proof}
Let $\nu\colon X' \rightarrow X$ be a foliated log resolution of $\cal F$.
Observe that $\nu$ will extract every $\mu$-exceptional divisor which is transverse to $\mu^{-1}\cal F$.
We may write $\nu^*K_{\cal F} +F = K_{\cal G}+\sum_{i = 1}^k E_i$ where $\cal G = \nu^{-1}\cal F$,
where the $E_i$ are the non $\cal G$ invariant exceptional divisors and where $F \geq 0$.  By \cite[Corollary 2.26]{Spicer17} 
we may run a $K_{\mathcal G}$-MMP
over $X$, call it $\phi\colon X' \rightarrow X''$, set $\mathcal H = \phi_*\mathcal G$ and let $\rho\colon X'' \rightarrow X$ be the
induced map.  
Only curves tangent to $\mathcal G$ will be contracted by this MMP, 
and so no component of $\sum_{i= 1}^k E_i$ will be 
contracted.
 
By foliation adjunction we see that 
$(K_{\cal H}+E'_i)\cdot E'_i \geq 0$ where $E'_i = \phi_*E_i$, see ~\cite[Proposition 3.4]{Spicer17} 
(note that in the notation of \cite[Proposition 3.4]{Spicer17} the resticted foliation $\cal H_{E'_i}$ 
is the foliation by points on $E'_i$ and so $K_{\cal H_{E'_i}} = 0$).  In particular, $K_{\mathcal H}+\sum E'_i$
is nef over $X$.
By the negativity lemma, ~\cite[Lemma 3.39]{KM98}, we have $\phi_*F = 0$, and so $K_{\mathcal H}+\sum E'_i$ is numerically trivial over $X$.
Since $K_{\mathcal H}$ is nef over $X$ it likewise follows that $-\sum E'_i = K_{\mathcal H} - (K_{\mathcal H}+\sum E'_i)$
is nef over $X$.  Another application of ~\cite[Lemma 3.39]{KM98} then shows that ${\rm exc}(\rho) = \sum E'_i$,
and so $\sum E'_i$ is connected.  
This together with the inequalities 
$0 \leq (K_{\mathcal H}+E'_i)\cdot E'_i \leq (K_{\mathcal H}+\sum_{i=1}^k E'_i)\cdot E'_i = 0$ implies that $k=1$, as required.
\end{proof}

\begin{lemma}
\label{lem_lc_mmp}
Let $X$ be a normal projective surface with a rank one foliation $\mathcal F$ and $\Delta \geq 0$ such that
$(\cal F, \Delta)$ is log canonical.  Suppose that $K_{\mathcal F}$ is $\bb Q$-Cartier.
Let $R \subset \ccurves{X}$ be a $(K_{\cal F}+\Delta)$-negative extremal ray and let $C \subset X$ be a 
$\mathcal F$-invariant curve
such that
\begin{enumerate}
\item $[C] \in R$; and

\item $C$ contains a strictly log canonical singularity of $\mathcal F$.

\end{enumerate}

Then $X$ is covered by curves spanning $R$ and $\rho(X) = 1$.
In particular, $K_{\cal F}+\Delta$ is not pseudo-effective.
\end{lemma}
\begin{proof}
Since $(\cal F, \Delta)$ is log canonical by \cite[Remark 2.12]{Spicer17} we know 
that no component of $\Delta$ is $\mathcal F$-invariant.
In particular, $\Delta\cdot C \geq 0$.  So it follows that $K_{\cal F}\cdot C <0$.

Let $P \in C$ be a strictly log canonical singularity of $\mathcal F$.  
If $n\colon \overline{C} \rightarrow C$ is the normalization then \cite[Proposition 2.16]{CS20} implies that we may write
$n^*K_{\cal F} = K_{\overline{C}}+\Theta$ where $\Theta \geq 0$ and
$\lfloor \Theta\rfloor$ is supported
exactly on the preimage of the non-terminal points of $\cal F$ contained in $C$.  
In particular,
since $K_C+\Theta<0$ it follows that for all other $Q \in C$, $Q \neq P$ that $\cal F$
is terminal at $Q$.

To see that $C$ moves, we may freely replace $X$ by a smaller open neighborhood of $C$ so that 
$\cal F$ is strictly log canonical at only $P$. 
Let $f\colon X' \rightarrow X$ be an F-dlt modification, which exists by~\cite[Theorem 1.4]{CS18}.
Thus, $K_{\cal F'} +E = f^*K_\cal F$, where $\cal F'= f^{-1}\cal F$ and $E$ is the unique irreducible $f$-exceptional divisor which is not $\cal F'$-invariant, see Lemma~\ref{lem_uniq_lc_place}, which implies that $K_{\cal F'}\cdot C'<0$, where $C'$ is the strict transform of $C$.
Let us observe that $(\cal F', E)$ is F-dlt, in particular log canonical, in a neighborhood of $E$.  
But, since $\cal F'$ is non-dicritical, then for any divisor $F$ centred over a point in a neighborhood of $E$, $a(F, \cal F', E) \geq \iota(F) = 0$, which in turn implies that $\cal F'$ is terminal in a neighborhood of $E$.  
Hence, $\mathcal F'$ is terminal at $C' \cap E$. 
As $(K_{\cal F'} +E) \cdot C' < 0$, then, by adjunction~\cite[Lemma 8.9]{Spicer17} $\cal F'$ is terminal at each point of $C'$.  
By Reeb stability~\cite[Proposition 3.3]{CS20}, $C'$ moves in family covering $X'$, and hence $C$ moves in a family covering $X$.

Finally, we claim that $C^2>0$.  
This follows because $(C')^2 = 0$, $C'\cdot E >0$ and $f^*C = C'+aE+F$ where 
$a>0$ and $F \geq 0$ is $f$-exceptional.  Thus $C$ is a big divisor and so $C$, and hence $R$, is contained in
the interior of $\overline{NE}(X)$. As $R$ is also an extremal ray of $\overline{NE}(X)$, then $\rho=1$.
\end{proof}

\begin{theorem}
\label{thm_recall_surf_mmp}
Let $X$ be a projective klt surface, let $\Delta \geq 0$ and let $\mathcal F$ a rank one foliation such that
$(\mathcal F, \Delta)$ and $(X, \Delta)$ are log canonical.

Let $R \subset \overline{NE}(X)$ be a 
$(K_{\mathcal F}+\Delta)$-negative extremal ray. Then there exists a $\mathcal F$-invariant curve $C$ such that $R= \mathbb{R}_{>0}[C] \subset \overline{NE}(X)$.

Moreover, there exists a contraction
$c_R\colon X \rightarrow Y$
contracting exactly those curves in $X$ whose numerical classes are contained in $R$ and such that the following conditions holds:
\begin{enumerate}
\item if $c_R\colon X \rightarrow Y$ is birational then $c_R$ contracts only $\cal F$-invariant curves;

\item if $c_R\colon X \rightarrow Y$ is a fibre type contraction then $R$ is $(K_X+\Delta)$-negative;

\item if there is a strictly log canonical singularity of $\mathcal F$ on $C$ then $\rho(X) = 1$, $-(K_{\mathcal F}+\Delta)$ and $-(K_X+\Delta)$ are ample, and $Y$ is a point.
\end{enumerate}
Moreover, in all cases the relative Picard number of the contraction $ = 1$.

In particular, we may run a 
$(K_{\mathcal F}+\Delta)$-MMP.
\end{theorem}
\begin{proof}
By the cone theorem for surface foliations, see \cite[Theorem 6.3 and Remark 6.4]{Spicer17}, a $(K_{\mathcal F}+\Delta)$-negative extremal ray $R \subset \overline{NE}(X)$ is spanned by the class of a curve $C$ which is $\mathcal F$-invariant.

Let us consider the case where $C^2 \geq 0$. 
By \cite[Theorem~6.3]{Spicer17}, there exists a nef Cartier divisor $H_R$ such that $H_R \cdot R =0$
and $H_R$ is positive on every other extremal ray.  Since $C^2 \geq 0$ it follows that $H_R$ cannot be big, i.e., $H_R^2 = 0$.
Arguing as in the proof of~\cite[Theorem 6.3]{Spicer17}, using~\cite[Corollary 2.28]{Spicer17}, it follows that $X$ is covered by a family of rational curves tangent to $\mathcal F$ whose numerical class spans $R$.

\medskip

{\bf Claim 1}.
{\it Let $\Sigma$ be a general choice of such a curve, then $(K_X+\Delta)\cdot \Sigma<0$}.

\begin{proof}[Proof of Claim 1]
Let $p\colon Y \rightarrow X$ be an F-dlt modification, which exists by~\cite[Theorem 1.4]{CS18}.
If $\mathcal G = p^{-1}\mathcal F$ and $\Delta' = p_*^{-1}\Delta$,
then $K_{\mathcal G}+\Delta' +\sum_i \iota(E_i)E_i = p^*(K_{\mathcal F}+\Delta)$,
where $E_i$ are the $p$-exceptional divisors.  
By construction, $\mathcal G$ is non-dicritical;
hence, it is induced by a fibration $Y \rightarrow B$, such that $\Sigma' = p_*^{-1}\Sigma$
is a fibre of $Y \rightarrow C$.  
It follows that $K_Y\cdot \Sigma' = K_{\mathcal G}\cdot \Sigma' = -2$.  
We may write $K_{Y}+\Delta' +\sum_i a_iE_i = p^*(K_X+\Delta)$, where $a_i \leq 1$ since $(X, \Delta)$ is log canonical. 
If $\iota(E_i) = 0$, then $E_i \cap \Sigma' = \emptyset$, since $\Sigma'$ is general, and so $E_i\cdot \Sigma' = 0$. 
Hence, $(\sum_i \iota(E_i)E_i)\cdot \Sigma' \geq (\sum_i a_iE_i)\cdot \Sigma'$, and
\begin{align*}
0> (K_{\mathcal F}+\Delta)\cdot \Sigma = 
& (K_{\mathcal G} +\Delta'+\sum_i \iota(E_i)E_i)\cdot \Sigma' \\
\geq &(K_{Y}+\Delta' +\sum_i a_iE_i)\cdot \Sigma' = 
(K_X+\Delta)\cdot \Sigma.
\end{align*}
\end{proof}
In view of Claim 1, then the existence of $c_R$ is immediate, since $c_R$ can be constructed as the contraction of a $(K_X+(1-\epsilon)\Delta)$-negative extremal ray $0<\epsilon \ll 1$, see~\cite[Theorem 3.7]{KM98}. 
Since $c_R$ is a fibre type contraction only if $C^2 \geq 0$, then item (2) above also follows at once.

\medskip

Item (3) follows from Lemma \ref{lem_lc_mmp} and (2).

\medskip

Now consider the case where $C^2<0$. 
In this case, the contraction $c_R$, if it exists, will be birational and it will only contract $C$, which proves (1).

From item (3), we know that
$\mathcal F$ has canonical singularities in a neighborhood of $C$.  We may apply~\cite[III.1-2]{McQuillan08} to contract $C$ -- strictly speaking, in~\cite{McQuillan08} an entire chain of rational curves is contracted, but the arguments provided work equally well to contract a single $K_{\mathcal F}$-negative curve.  
For the reader's convenience we will supply an alternate proof of the existence of this contraction.  
By~\cite[Theorem 11.3]{CS18}, $\mathcal F$ has non-dicritical singularities in a neighborhood of $C$ and so by~\cite[Lemma 8.14]{Spicer17}\footnote{Let us observe that both~\cite[Theorem 11.3]{CS18} and~\cite[Lemma 8.14]{Spicer17} are stated for threefolds. We can deduce the analogous statement for rank one foliation on surfaces by applying the results to the threefold $X \times B$ where $B$ is a smooth curve, and to the foliation $\pi^{-1}\cal F$ where $\pi\colon X \times B \rightarrow X$ is the projection.} this implies that $(X, \Delta+C)$ is a log canonical pair. 

\medskip

{\bf Claim 2}. $(K_X+\Delta+C)\cdot C<0$.  
\begin{proof}[Proof of Claim 2]
Let $p\colon Y \rightarrow X$ be an 
F-dlt modification, which can be performed by \cite[Theorem 1.4]{CS18}.
Let $C'$ be the strict transform of $C$ and let $\Delta' = p_*^{-1}\Delta$.
We write $(K_Y+\Delta'+C'+\sum_iE_i)= p^*(K_X+\Delta+C)+\sum_i b_iE_i$, where the $E_i$ are the $p$-exceptional prime divisors and $b_i \geq 0$.  
Since $\mathcal F$ is non-dicritical, then for all $i$, $E_i$ is invariant; thus, $K_{\mathcal G}+\Delta' = p^*(K_{\mathcal F}+\Delta)$.
Since $\mathcal G$ is F-dlt,~\cite[Lemma 3.12]{CS18} implies that $\mathcal G$ is terminal at the singular points of $X$.  We may then apply \cite[Lemma 8.9]{Spicer17}
to conclude that 
\begin{align*}
(K_{\mathcal G}+\Delta')\cdot C' \geq (K_Y+\Delta'+C'+\sum_iE_i)\cdot C' \geq (K_X+\Delta+C)\cdot C.
\end{align*} 
Again, here the cited
results are stated for threefolds but apply equally well to surfaces as explained above.
\end{proof}
We conclude by observing that $(X, (1-\epsilon)(\Delta+C))$ is klt for all $0<\epsilon$ and $(K_X+(1-\epsilon)(\Delta+C))\cdot C<0$
for $0<\epsilon \ll 1$ and so we may contract $C$ by a $(K_X+(1-\epsilon)(\Delta+C))$-negative extremal contraction, 
\cite[Theorem 3.7]{KM98}.

Since all our contractions are 
$(K_X+\Theta)$-negative 
contractions for a klt pair 
$(X, \Theta)$,~\cite[Corollary 3.17]{KM98} 
implies that they are of relative Picard number one.

\medskip

Finally, it is a standard argument to show that the existence of divisorial contractions as explained above implies the existence
of the $(K_{\mathcal F}+\Delta)$-MMP.
\end{proof}

\begin{remark}
\label{remark_dlt_preserved}
Let notation be as in Theorem \ref{thm_recall_surf_mmp}.
The above proof shows that if 
$\sum_i C_i$ is any collection of reduced $\mathcal F$-invariant curves such that $\mathcal F$ has canonical singularities in a neighborhood of $\sum_iC_i$, then each step of the $K_\mathcal{F}$-MMP
is also a step of the $(K_X+\Delta+\sum C_i)$-MMP.  
In particular, if $(X, \Delta+\sum C_i)$ is dlt and $\phi\colon X \rightarrow X'$ is a run of the $(K_{\mathcal F}+\Delta)$-MMP, then $(X', \phi_*(\Delta+\sum C_i))$ is again dlt.
\end{remark}

\begin{lemma}
\label{lem_resolution}
Let $X$ be a normal surface, let $D$ be a reduced Weil divisor, and let $\cal F$ be a rank one foliation on $X$ such that
\begin{enumerate}
\item $K_{\cal F}$ is Cartier; and

\item every component of $D$ is $\cal F$-invariant.
\end{enumerate}
Then there exists a log resolution of $(X, D)$ which only extracts
divisors $E$ of foliation discrepancy $\leq -\iota(E)$.

\end{lemma}
\begin{proof}
The problem is local about any point $P \in X$ so we may freely assume $(X, D)$ is not log smooth at $P$ 
and that $\mathcal F$ is generated by a vector field $\partial$.  Since $(X, D)$ is not log smooth at $P$
either $X$ or $D$ is singular at $P$, and so by \cite[Lemma 2.6]{CS20} $P$ is invariant under $\partial$.   
By \cite[Lemma 1.1.3]{bm16} if $b\colon \tilde{X} \rightarrow X$ is the blow up in $P$, then $\partial$ 
lifts to a vector field $\tilde{\partial}$ on $\tilde{X}$, which moreover leaves $b^{-1}(E)$ invariant.

A log resolution of $(X, D)$ may be acheived by repeatedly blowing up centers where $(X, D)$ is not log smooth, so
by applying the above observation and arguing by induction on the number of blow ups in a log resolution 
we may produce a log resolution $\pi\colon X' \rightarrow X$ and a lift $\partial'$ of $\partial$ which leaves $\pi^{-1}(P)$
invariant.

Because $\partial'$ leaves $\pi^{-1}(P)$ invariant we see that if $F$ is a $\pi$-exceptional
divisor with $\iota(F) = 1$ then $\partial'$ vanishes along $F$.
Our first claim then follows by observing that the foliation discrepancy along a divisor $F$
is exactly $-a$ where $a$ is the order of vanishing of $\partial'$ along $F$.
\end{proof}

\begin{lemma}
\label{lem_strictly_lc_goren}
Let $P \in X$ be a germ of a normal surface and let $\cal F$ be a rank one foliation on $X$.
Suppose that $\cal F$ is strictly log canonical at $P$.  Then $\cal F$ is Gorenstein at $P$.
\end{lemma}
\begin{proof}
Let $\sigma\colon P' \in X' \rightarrow P \in X$ be the index one cover associated to $K_{\cal F}$
with Galois group $G \cong \bb Z/m \bb Z$, let $\cal F' = \sigma^{-1}\cal F$ and let $\partial$ be a vector field defining $\cal F'$.
By Lemma \ref{lem_cover2} we have $\cal F'$ is strictly log canonical at $P'$.

Let $\ff m$ denote the maximal ideal of $X'$ at $P'$. Since $\cal F'$ is strictly log canonical by \cite[Fact III.i.3]{McQP09} 
(up to renormalizing $\partial$ by a constant in $\bb C$) 
we may write the linear part of $\partial$ as $\partial_0 = \sum n_i x_i\frac{\partial}{\partial x_i}$
where $n_1, ..., n_k$ are positive integers and $x_1, ..., x_k \in \ff m$ give a basis of $\ff m/\ff m^2$.

Let $g \in G$.  On one hand, by assumption 
we may write $g_*\partial = \zeta \partial$ where $\zeta$ is a primitive $m$-th root of unity.
On the other hand,  from the equality $g^*(g_*\partial(x)) = \partial(g^*x)$ for any $x \in \ff m$
we see that linear parts of $g_*\partial$ and $\partial$ have the same
eigenvalues.
It follows that $\{n_1, ..., n_k\} = \{\zeta n_1, ..., \zeta n_k\}$
and so $\zeta = 1$, i.e., $K_{\cal F}$ was Gorenstein
to begin with.
\end{proof}

\subsection{$(\epsilon, \delta)$-adjoint log canonical foliated singularities}
\label{s_ed_alc_defns}
We wish to measure singularities of triples $(X, \cal F, \Delta)$ in terms of  how $K_{(X, \cal F, \Delta), \epsilon}$ changes under birational transformations.
This idea was initially considered in \cite{PS16}, see Remark~\ref{rmk_different_notation} below, and the approach here is a natural extension of the ideas introduced there.

\begin{defn}
\label{def:alc.sings}
Let  $(X, \cal F, \Delta)$ be a foliated triple. 
Fix $\epsilon>0$ and  $\delta\geq 0$.
\\
We say that $(X, \cal F, \Delta)$ is {\bf $(\epsilon, \delta)$-adjoint log canonical} 
(resp. {\bf $(\epsilon, \delta)$-adjoint klt})
provided that for any birational morphsim 
$\pi\colon X' \rightarrow X$ if we write
\begin{align*}
(K_{\cal F'}+\Delta'_{\text{n-inv}})+
\epsilon (K_{X'}+\Delta') = 
\pi^*((K_{\cal F}+\Delta_{\text{n-inv}})+
\epsilon (K_X+\Delta))+E,
\end{align*} 
where
\begin{itemize}
	\item
 $E = \sum a_iE_i$ is $\pi$-exceptional, and
 	\item
 $\Delta'\coloneqq \pi^{-1}_\ast \Delta$,
\end{itemize}
then for all $i$
\begin{align*}
a_i \geq & (\iota(E_i)+\epsilon)(-1+\delta) \\
(\text{resp.}, \ \lfloor \Delta \rfloor = 0 \ \text{and} 
\ a_i > &  (\iota(E_i)+\epsilon)(-1+\delta)).
\end{align*}
\end{defn}

When $\delta =1$ we will refer to $(\epsilon, \delta)$-\alc as $\epsilon$-adjoint canonical.
When $\delta = 0$ we will refer to $(\epsilon, \delta)$-\alc as $\epsilon$-adjoint log canonical.

\begin{remark}
\label{rmk_different_notation}

In \cite{PS16} the notion of {\it $\epsilon$-canonical} was defined whereby singularities were measured by considering how the adjoint series $K_{\cal F}+\epsilon N^*_{\cal F}$ transforms under blow ups.

By re-writing 
\begin{align*}
K_{\cal F}+\epsilon N^*_{\cal F} = (1-\epsilon)K_{\cal F}+\epsilon K_X,
\end{align*}  
then it is immediate to see that if a foliated pair $(X, \mathcal F)$ is $\frac{\epsilon}{1+\epsilon}$-canonical in the sense of \cite{PS16}, then it is also $(\epsilon, 1)$-\alc in the sense of Definition~\ref{def:alc.sings}.
For various computations we need to perform, working with $K_{\cal F}+\epsilon K_X$ seemed preferable to us
hence the slight change in the definition.
\end{remark}

The following three lemmata are immediate consequences of the definition of $(\epsilon, \delta)$-\alc singularities.

\begin{lemma}
\label{lem_perturb_epsilon}
Let $(X, \cal F, \Delta)$ be a foliated triple and $\epsilon >0$.
\begin{enumerate}
\item \label{lem_perturb_epsilon:part1}
Let $0 \leq \delta'< \delta \in [0,1]$.
If $(X, \cal F, \Delta)$ is $(\epsilon, \delta)$-\alc then it also $(\epsilon, \delta')$-\alc.
\item 
Assume that $\Delta = 0$ and that $(X, \cal F)$ is $\epsilon$-\acanonical.
\begin{enumerate}
\item If $\cal F$ has canonical singularities, then $(X, \cal F)$ is $\epsilon'$-\acanonical for all $0 < \epsilon' \leq \epsilon$.

\item If $X$ has canonical singularities, then $(X, \cal F)$ is $\epsilon''$-\acanonical for all $\epsilon'' \geq \epsilon$.
\end{enumerate}
\end{enumerate}
\end{lemma}
\begin{proof}
Follows from a direct computation.
\end{proof}

\begin{lemma}
\label{lem_easy_ed_crit}
Let $X$ be a smooth surface 
and let $0 \leq \Delta$ be $\bb Q$-divisor with snc support such that $\lfloor \Delta \rfloor = 0$.
Let $\cal F$ be a rank one foliation such that $(\cal F, \Delta_{\text{n-inv}})$ is canonical.
Let $\delta>0$ be such that all coefficients
of  $\Delta$ are  $\leq 1-\delta$.

Then
$(X, \cal F, \Delta)$ is $(\epsilon, \delta)$-\alc for all $\epsilon>0$.
\end{lemma}
\begin{proof}
Follows from a direct computation.
\end{proof}

\begin{lemma}
\label{lem_ed_sing_pres}
Let $p\colon X \rightarrow Y$ be a birational morphism between surfaces, let $(\cal F, \Delta)$ be a foliated pair on $X$ where $\cal F$ is of rank one 
and let $\cal G \coloneqq p_\ast \cal F$.
Fix $\epsilon>0$ and  $\delta \geq 0$.
Assume that
\begin{enumerate}
\item the coefficients of $\Delta$ are $\leq 1-\delta$;

\item $-K_{(X, \cal F, \Delta), \epsilon}$ is $p$-nef; and

\item $(X, \cal F, \Delta)$ is $(\epsilon, \delta)$-\alc.
\end{enumerate}

Then $(Y, \cal G, p_*\Delta)$ is $(\epsilon, \delta)$-\alc.
\end{lemma}
\begin{proof}
This is a direct consequence of the negativity lemma, \cite[Lemma 3.38]{KM98}.
\end{proof}

Our goal for the last part of this subsection is to prove the following generalization of~\cite[Proposition 4.9]{PS16}.
\begin{proposition}
\label{cor_bound2}
Let $I \subset [0, 1]$ be a subset satisfying the DCC.
Then there exists a positive real number $E=E(I)$ such that the following statement holds:

Let $0 \in X$ be a germ of a klt surface singularity,
let $\cal F$ be a rank one foliation on $X$ and let $\Delta \geq 0$
be a divisor with $\Delta \in I$ so that
$(X, \cal F, \Delta)$ is $\epsilon$-\alc for some $0<\epsilon<E(I)$.
Then $(\cal F, \Delta_{\text{n-inv}})$ is log canonical.
\end{proposition}

We start by proving three ancillary lemmata that will be used in the proof of the proposition.

\begin{lemma}
\label{lem_refined_discrep_calc}
Let $\partial$ be a germ of a vector field on $P \in \bb C^2$ and suppose that $\partial$
is singular at $P$ and the linear part of $\partial$ at $P$ is $=0$.
Let $\pi\colon X \rightarrow \bb C^2$ be the blow up at $P$ with exceptional divisor $E$
and let $\cal F$ be the foliation generated by $\partial$.
Then
\begin{align*} K_{\pi^{-1}\cal F} = \pi^*K_{\cal F}-bE,\end{align*}
where $b \geq \iota(E)+1$.
\end{lemma}
\begin{proof}
Using notation as in \cite[Chapter 1, \S 2]{Brunella00}
let $\omega$ be a one form with an isolated zero at $P$ and $\partial(\omega) = 0$.
Let $a(P)$ denote the order of vanishing of $\omega$ at $P$
and let $l(P)$ denote the order of vanishing of $\pi^*\omega$ along $E$.

A direct computation shows 
that $l(P) = a(P)$ when $E$ is invariant
and $l(P) = a(P)+1$ when $E$ is not invariant.  Another straightforward calculation shows
that the discrepancy of our blow up is $-(l(P)-1) = -(a(P)+\iota(E)-1)$ and by assumption $a(P) \geq 2$
from which our claim follows.
\end{proof}

\begin{lemma}
\label{lem_e-lc_thresh}
For $0< \epsilon< \frac{1}{5}$ 
the following holds.

Let $\cal F$ be a rank one foliation on $P \in \bb C^2$ and suppose that $(\bb C^2, \cal F)$
is $\epsilon$-\alc at $P$.  Then $\cal F$ is log canonical at $P$.
\end{lemma}
\begin{proof}
Suppose that $\cal F$ is not log canonical at $P$.

Following the proof of \cite[Proposition 4.9]{PS16} (see also the proof of \cite[Theorem 1.1]{Brunella00})
we may find a sequence of at most 3 blow-ups
$b_i\colon (X_i, \cal F_i) \rightarrow (X_{i-1}, \cal F_{i-1})$
such that
\begin{enumerate}
\item $(X_0, \cal F_0)\coloneqq (\bb C^2, \cal F)$;
\item $b_i$ is a blow up in the singular locus of $\cal F_{i-1}$; and 
\item on the last blow up, call it $b_{n}$, we blow up a foliation singularity whose linear part is $=0$.
\end{enumerate}

Let $\pi\colon X_n \rightarrow \bb C^2$ denote the composition of these blow ups and let $E$ be the exceptional divisor of $b_n$.
By Lemma \ref{lem_refined_discrep_calc} 
$a(E, \cal F) \leq -(\iota(E)+1)$, 
and a direct computation shows that 
$a(E, X) \leq 4$.  
Thus, 
\begin{align*}
a(E, \cal F)+\epsilon a(E, X) \leq -(\iota(E)+1)+4\epsilon.
\end{align*}
For all $\epsilon <\frac{1}{5}$, 
$-(\iota(E)+1)+4\epsilon < -(\iota(E)+\epsilon)$, which in turn implies that 
$a(E, \cal F)+\epsilon a(E, X)< -(\iota(E)+\epsilon)$, contradicting the assumption that 
$(X, \cal F)$ is $\epsilon$-\alc.
\end{proof}

We remark that if $\epsilon>0$ then an $\epsilon$-\acanonical singularity
is not necessarily a canonical foliation singularity, i.e., it could be strictly
log canonical.
Moreover if we fix $\epsilon>0$
then a terminal singularity of $\cal F$ is not necessarily $(\epsilon, 1)$-\alc.
Consider for instance the quotient of $(\bb C^2, \langle \frac{\partial}{\partial x}\rangle)$
by the action of $\bb Z/m\bb Z$ given by $(x, y) \mapsto (\xi x, \xi^b y)$
where $(m, b) = 1$.
This will always gives a terminal foliation singularity, however,
it is not $\epsilon$-\acanonical for $\epsilon> \frac{1}{m-2}$. 

\begin{lemma}
\label{lem_cover2}
Let $X$ be a normal variety and let $\cal F$ be a rank one foliation on $X$. Let $\sigma\colon X' \rightarrow X$
be a finite morphism, \'etale in codimension 1 and set $\cal F' \coloneqq \sigma^{-1}\cal F$.
Write $\Delta' = \sigma^*\Delta$ and note that $\Delta'_{\text{n-inv}} = \sigma^*\Delta_{\text{n-inv}}$.
Fix $\epsilon>0$ and $\delta \geq 0$.

\begin{enumerate}
\item Suppose that $(X, \cal F, \Delta)$ is $(\epsilon, \delta)$-\alc. Then $(X', \cal F', \Delta')$ 
is $(\epsilon, \delta)$-\alc.

\item Suppose that $(\cal F', \Delta'_{\text{n-inv}})$ is (log) canonical.  Then $(\cal F, \Delta_{\text{n-inv}})$ 
is (log) canonical.
\end{enumerate}

\end{lemma}
\begin{proof}
The proof of item (1) and the log canonical case of item (2) is essentially identical to the proof of \cite[Proposition 5.20]{KM98}, making
use of the statement of foliated Riemann-Hurwitz found in \cite[Lemma 3.4]{Druel18}
which gives us the following adjoint Riemann-Hurwitz formula
for a finite morphism $\sigma\colon Y \rightarrow X$ for all $\epsilon \geq 0$:
\begin{align*}
K_{(Y, \cal G, \sigma^*\Delta), \epsilon} = \sigma^*K_{(X, \cal F, \Delta), \epsilon}+\sum_{D \in \text{Div}(Y)} (r_D-1)(\iota(D)+\epsilon)D,
\end{align*}
where $\cal G = \sigma^{-1}\cal F$ and $r_D$ is the ramification index of $\sigma$ along $D$.

We now explain the proof of item (2) in the canonical case.
Let $f\colon W \rightarrow X$ be a birational morphism,
and consider the following diagram
\begin{align*}
\xymatrix{
W' \ar[d]_{g} \ar[r]^{\tau} & W \ar[d]^{f} \\
X' \ar[r]^{\sigma} & X
}
\end{align*}

where $W'$ is the normalization of the main component of $W\times_XX'$.  Set 
$\mathcal H \coloneqq f^{-1}\mathcal F$ and $\mathcal H' \coloneqq g^{-1}\mathcal F'$.  
Let $E \subset W$ be an exceptional divisor, let $E'$ be a component of $\tau^{-1}(E)$ and let $r$ be the ramification index of $\tau$
along $E'$.
By assumption $\mathcal F'$ has canonical singularities and so by \cite[Corollary III.i.4]{McQP09} $\mathcal F'$ is non-diciritcal, hence every $g$-exceptional divisor is 
$\mathcal H'$-invariant.  Thus, every $f$-exceptional divisor is $\mathcal H$-invariant.
Doing a calculation analogous to the one in \cite[Proposition 5.20]{KM98}
and making use of the foliated Riemann-Hurwitz formula, \cite[Lemma 3.4]{Druel18},
and noting that $\iota(E') = 0$ gives us
\begin{align*}a(E', \mathcal F', \Delta') = ra(E, \mathcal F, \Delta).\end{align*}
Since $a(E', \mathcal F', \Delta') \geq 0$, it follows that $a(E, \mathcal F, \Delta) \geq 0$ as required.
\end{proof}

\begin{proof}[Proof of~\ref{cor_bound2}]
Since $0 \in X$ is a klt singularity it is a quotient singularity and so
by Lemma \ref{lem_cover2} we may freely reduce to the case where $X = \bb C^2$.
We may also assume, without loss of generality, that $\Delta = \Delta_{\text{n-inv}}$ and by
taking $\epsilon <\frac{1}{5}$ and applying 
Lemma \ref{lem_e-lc_thresh}
we may assume that $\cal F$ is log canonical at $0$.

We now proceed by arguing in cases, based on whether or not $\cal F$ is singular at $0$.

\medskip

{\bf Case 1}. \emph{We assume that $\cal F$ is singular at $0$.}

Since $\cal F$ is log canonical at $0$ it suffices to show 
that for $\epsilon$ sufficiently small if $(X, \cal F, \Delta)$ 
is $\epsilon$-\alc then $\Delta$ is disjoint from $0$.

So, suppose that $0$ is in the support of $\Delta$, let $\pi\colon X \rightarrow \bb C^2$ be the blow up
at $0$ with exceptional divisor $C$, let $i_0$ be the smallest strictly positive element of $I$,
let $\cal G = \pi^{-1}\cal F$ and let $\Delta' = \pi_*^{-1}\Delta$.

Since $\cal F$ is log canonical and singular at $0$
we have that the blow up at $0$
has foliation discrepancy $=-\iota(C)$ and so
 $K_{\cal G}+\Delta' = \pi^*(K_{\cal F}+\Delta) +aC$ where $a \leq -\iota(C)-i_0$
and $K_X+\Delta' = \pi^*(K_{\bb C^2}+\Delta) +bC$ where $b \leq 1-i_0$.
It follows that
$K_{(X, \cal G,\Delta'), \epsilon} = \pi^*(K_{(\bb C^2, \cal F,\Delta), \epsilon})+(a+\epsilon b)C$
where $a+\epsilon b \leq -\iota(C)-i_0+\epsilon (1-i_0)$.

Now $(\bb C^2, \cal F, \Delta)$ will fail
to be $\epsilon$-\alc if 
we have the inequality $-\iota(C)-i_0 +\epsilon(1-i_0) < -(\iota(C)+\epsilon)$
or equivalently
$\frac{\iota(C)+i_0-\epsilon(1-i_0)}{\iota(C)+\epsilon} >1$.
However, this inequality will hold for all $\epsilon$ smaller 
than some constant depending
only on $i_0$.

\medskip

{\bf Case 2}. 
\emph{We assume that $\cal F$ is smooth at $0$.}

Let $0 \in L$ be a germ of a leaf through
$0$ and observe that the discrepancies of $(\cal F, \Delta)$ are exactly the log discrepancies
of $(X, \Delta+L)$, see \cite[Lemma 8.14]{Spicer17} and its proof (again we remark that the cited Lemma is proven for threefolds, but holds
for surfaces as well).  
Thus, to show that $(\cal F, \Delta)$ is log canonical 
it suffices to show that $(X, \Delta+L)$ is log canonical.  
\\
We now claim that $(X, \cal F, \Delta)$ being $\epsilon$-\alc implies
that $(X, \Delta+\frac{1}{1+\epsilon}L)$ is log canonical.
Indeed, let $\pi\colon Y \rightarrow X$ be any birational morphism
and let $\Delta_Y$ and $L_Y$ denote the strict transforms of $\Delta$ and $L$ respectively
and let $C = \sum C_i$ be the sum of the $\pi$-exceptional divisors with coefficient $=1$.  Note 
that all the $C_i$ are invariant.
Thus, by definition of $\epsilon$-\alc and the above observation we may write
\begin{align*}(K_{Y}+\Delta_Y+L_Y+C)+ \epsilon(K_Y+\Delta_Y) = \pi^*((K_X+\Delta+L)+\epsilon (K_X+\Delta))+\sum a_iC_i\end{align*}
where $a_i \geq -\epsilon$.
Dividing the above equality by $1+\epsilon$
gives
\begin{align*}K_Y+\Delta_Y+\frac{1}{1+\epsilon}(L_Y+C) = \pi^*(K_X+\Delta+\frac{1}{1+\epsilon}L)+\sum \frac{a_i}{1+\epsilon}C_i.\end{align*}
Note that we always have the inequality $\frac{1-a_i}{1+\epsilon} \leq 1$ for all $i$ and so
it follows that $(X, \Delta+\frac{1}{1+\epsilon} L)$ is log canonical.
\\  
Let $\lambda$ be the log canonical threshold of $(X, \Delta)$ with respect to $L$.
We have just shown $\lambda \geq \frac{1}{1+\epsilon}$.

By the ACC for log canonical thresholds, \cite[Theorem 1.1]{HMX12},
we see that there is a fixed $\lambda_0$ depending only on $I$ so that if 
$\lambda \geq \lambda_0$ then $\lambda = 1$, in which case $(X, \Delta+L)$ is log canonical, which implies
that $(\cal F, \Delta)$ is log canonical.
\medskip

Choose $E(I)$ 
so that $\epsilon<E(I)$ implies 
\begin{enumerate}
\item $\frac{1}{1+\epsilon} \geq \lambda_0$;
\item $\epsilon <\frac{1}{5}$; and
\item $\frac{\iota(C)+i_0-\epsilon(1-i_0)}{\iota(C)+\epsilon} >1$.
\end{enumerate}
We therefore see that if
$\epsilon <E(I)$ then $(\cal F, \Delta)$ is log canonical.
\end{proof}

\subsection{A general boundedness result}

\begin{lemma}
\label{lem_main_bound}
Fix positive real numbers $\eta, \theta$.
Let $X$ be a projective $\eta$-lc variety
of dimension $n$ 
and let $N$ be a $\mathbb{R}$-divisor on $X$ such that
\begin{enumerate}
\item 
$N$ is nef and big;

\item
$N-K_X$ is pseudo-effective; and

\item
$N= P+E$ with $P$ integral and pseudo-effective, 
and $E \geq 0$ is effective and all its non-zero coefficients are $\geq \theta$.

\end{enumerate}
Then there exists an $m= m(\dim(X), \eta, \theta)$
such that for any $m'\geq m$, $|\lfloor m'N \rfloor|$ defines a birational map.
\end{lemma}
\begin{proof}
This is~\cite[Theorem~4.2]{Bir20}
\end{proof}

\subsection{Boundedness and foliations}

We make note of a simple result on the boundedness of foliations in families.

We recall that a bounded family of proper normal surfaces is a proper and flat morphism $f \colon \mathcal X \to T$ of finite type varieties such that any fibre of $f$ is a normal surface.
When $T$ is smooth (but not necessarily connected), then each connected component of $\mathcal X$ is normal.

Let $X$ be a normal variety.  
By a Weil divisorial sheaf $\mathcal K$ we mean a sheaf of the form $\mathcal O_X(K)$ where $K$
is a Weil divisor on $X$.  
By the support of $\mathcal K$ we mean the support of $K$ as a Weil divisor.

\begin{lemma}
\label{lem:bounded.fols}
Let $f \colon \mathcal X \to T$ be a bounded family of normal surfaces and let $\mathcal K$ be a Weil divisorial sheaf on $\mathcal X$ whose support does not contain any fiber of $f$.

Assume that the following hold:
\begin{enumerate}
\item $T$ is smooth;

\item $f\colon \mathcal X \to T$ is flat; and

\item for all $t \in T$ the restriction $\mathcal K\vert_{X_t}$ is reflexive. 
\end{enumerate}

Then there exists a bounded family of normal surfaces $f' \colon \mathcal X' \to T'$ and a rank one foliation $\mathcal F$ on $\mathcal X'$ which is tangent to $f'$ satisfying the following condition.

For all $t \in T$ and for any foliation $\mathcal G_t$ on $\mathcal X_t$ of canonical divisor $\mathcal K\vert_{\mathcal X_t}$, there exists $t' \in T'$ such that 
\begin{align*}
\mathcal X_t \simeq \mathcal X'_{t'}, 
\quad
\mathcal F\vert_{\mathcal X_{t'}} \simeq \mathcal G_t.
\end{align*}
\end{lemma}

\begin{proof}
Stratifying $T$ if necessary, we can assume that $f_\ast T_{\mathcal{X}/T}(\mathcal K)$ is a locally free sheaf, by making the natural map 
\begin{align*}
H^0(\mathcal X, T_{\mathcal{X}/T}(\mathcal K))) \to 
H^0(\mathcal X_t,T_{\mathcal{X}_t}(\mathcal K\vert_{\mathcal X_t}))
\end{align*}
surjective.
Here $T_{\mathcal{X}/T}$ is defined by the standard exact sequence induced by the differential of $f$,
\begin{align*}
\xymatrix{
0 \ar[r] &
T_{\mathcal{X}/T} \ar[r] &
T_{\mathcal X} \ar[r] &
f^\ast T_T
}
\end{align*}
Then, for any connected component $\overline T$ of $T$, we denote by $V^0_{\overline{T}}$ the total space of the vector bundle $V_{\overline T}$ on $\overline T$ minus its $0$-section, where $V_{\overline T}$  is the vector bundle associated to the locally free sheaf
$\mathcal V_{\overline T} \coloneqq 
f_\ast T_{\mathcal{X}/T}(\mathcal K)\vert_{\overline T}$. 
We denote by 
$v_{\overline{T}} \colon V^0_{\overline{T}} \to \overline{T}$ 
the structure map and by 
$\sigma \in H^0(V^0_{\overline{T}}, v_{\overline{T}}^\ast(f_\ast T_{\mathcal{X}/T}(\mathcal K)\vert_{\overline T})$ 
the tautological section.
We have the following commutative fibre-square diagram
\begin{align*}
\xymatrix{
\mathcal X^0 \coloneqq \mathcal X \times_T \overline{V}_{\overline{T}}^0 
\ar[r]^(.7){p} \ar[d]_{f_{\overline{V}_{\overline{T}}^0} } &
\mathcal X\vert_{\overline T} 
\ar[d]^{f\vert_{\overline T}} 
\\
V^0_{\overline{T}} \ar[r]_{v_{\overline{T}}}
&
\overline T
}
\end{align*}
As $v_{\overline T}$ is smooth, then the same holds for $p$, so that the pullback $p^\ast \mathcal K$ is well-defined, $p^\ast \mathcal O(\mathcal K) \simeq \mathcal O(p^\ast \mathcal K)$, and there is a natural isomorphism 
\begin{align*}v_{\overline{T}}^\ast(f_\ast T_{\mathcal{X}/T}\mathcal K\vert_{\overline T})
\simeq 
(f_{\overline{V}_{\overline{T}}^0})_{\ast} T_{\mathcal X^0 / \overline{V}_{\overline{T}}^0}(p^\ast \mathcal K).
\end{align*}
Under such isomorphism, the tautological section $\sigma$ is mapped to the section 
$\sigma' \in H^0(\overline{V}_{\overline{T}}^0, (f_{\overline{V}_{\overline{T}}^0})_{\ast} T_{\mathcal X^0 / \overline{V}_{\overline{T}}^0}(p^\ast \mathcal K))$  
that at
$w=(\overline t, \overline s) \in V^0_{\overline{T}}$, 
$\overline t=v_{\overline T}(w) \in \overline T$, 
$\overline s \in H^0(\mathcal{X}_{\overline{t}},\overline{T}_{\mathcal{X}_{\overline{t}}}(\mathcal K\vert_{\mathcal X_{\overline{t}}}) \setminus \{0\}$,
naturally induces the injection 
\begin{align*}
\xymatrix{
0 \ar[r] &
-\mathcal K\vert_{\mathcal{X}_{\overline{t}}} \ar[r]^\sigma &
T_{\mathcal{X}_{\overline{t}}}.
}
\end{align*}
In particular, any foliation $\mathcal G_t$ in the statement of the lemma will appear as one of the foliations just constructed. 
\end{proof}

The above result justifies the following definition a bounded family of foliated triples.

\begin{defn}
A bounded family of d-dimensional foliated triples is the datum of a foliated triple 
$(\mathcal Y, \cal G, \Gamma)$, 
where $\cal G$ has rank one,
and a projective morphism 
$f \colon \cal Y \to T$ 
to a variety of finite type $T$ 
such that
\begin{enumerate}
\item 
$f$ is a family of foliations over $T$ 
(with respect to the foliation $\cal G$), 
in the sense of ~\cite[Definition~2.3]{Chen19}; and,

\item 
for all $t\in T$, 
$(\mathcal Y_t, \cal G_t, \Gamma_t)$
is a projective foliated triple with 
$\dim \cal Y_t=d$.
\end{enumerate}
\end{defn}

We will often use the streamlined notation 
$f \colon (\cal Y, \cal G, \Gamma) \to T$
to denote bounded families of triples.
Given one such bounded family of triples and 
$t \in T$, 
we will denote by 
$(\cal Y_t, \cal G_t, \Gamma_t)$ 
the projective foliated triple induced on the fibre over the point $t$.

We will also say that a collection 
$\mathfrak D$ 
of projective 
$d$-dimensional 
foliated triples is bounded (or, forms a bounded family), if there exists a bounded family of foliations
$f \colon (\cal Y, \cal G, \Gamma) \to T$
such that any triple 
$(X, \cal F, \Delta)$ 
appears as one of the fibers of family given by 
$f$.

\section{Adjoint MMP}

Let $X$ be a smooth projective surface and $\cal F$ a foliation with reduced singularities.
Under this assumption, it is known that we may run either one of a $K_X$-MMP or a $K_{\cal F}$-MMP starting at $X$, or at $(X, \cal F)$, respectively.
The goal of this section is to show that we may also run a $K_{(X, \cal F, \Delta), \epsilon}$-MMP
for $\epsilon>0$ sufficiently small, see Theorem~\ref{thm_adj_mmp}.
We will show, moreover, that the singularities that arise in the run of such MMP are relatively mild, cf. Corollary~\ref{cor_bounded_sing}.

\subsection{Running the $K_{(X, \cal F, \Delta), \epsilon}$-MMP}

We start by proving the existence and termination of the $K_{(X, \cal F, \Delta), \epsilon}$-MMP.

\begin{theorem}
\label{thm_adj_mmp}
Let $(X, \cal F, \Delta)$ be a projective foliated triple where $X$ is a surface, 
$(X, \Delta)$ is dlt
and  $\cal F$ is rank one.
Fix $\delta \geq 0$.
Suppose that $\Delta \in I \cap [0, 1-\delta]$ where $I\subset [0, 1]$ is a DCC set. 
Fix $0<\epsilon <E(I)$ where $E(I)$ is as in Proposition~\ref{cor_bound2}.
Suppose that $(X, \cal F, \Delta)$ is $(\epsilon, \delta)$-\alc.

Then we may run a 
$K_{(X, \cal F, \Delta), \epsilon}$-MMP
$\rho\colon X \rightarrow Y$. 
Moreover,  setting $\cal G~:=~\rho_*\cal F$, $\Gamma:= \rho_\ast \Delta$ the following properties hold:

\begin{enumerate}
\item $(Y, \cal G, \Gamma)$ is $(\epsilon, \delta)$-\alc.

\item $(\cal G, \Gamma_{\text{n-inv}})$ has log canonical singularities and $(Y, \Gamma)$ is dlt.

\item If $K_{(X, \cal F, \Delta), \epsilon}$ is pseudo-effective then
$K_{(Y, \cal G, \Gamma), \epsilon}$ is nef.

\item If $K_{(X, \cal F, \Delta), \epsilon}$ is not pseudo-effective then 
 $Y$ admits a fibration $f\colon Y \rightarrow Z$ with $\rho(X/Y) = 1$ and
such that $-K_{(Y, \cal G, \Gamma), \epsilon}$ is $f$-ample.
\end{enumerate}
\end{theorem}

\begin{proof}
First, let us note that Lemma~\ref{lem_perturb_epsilon}.\ref{lem_perturb_epsilon:part1} implies that 
$(X, \mathcal F, \Delta)$ is $\epsilon$-\alc;
thus, by our choice of $\epsilon$, Proposition \ref{cor_bound2} implies that
$(\cal F, \Delta_{n-inv})$ 
has log canonical singularities.
We now explain how to run the $K_{(X, \cal F, \Delta), \epsilon}$-MMP on $X$.
\\
If $K_{(X, \cal F, \Delta), \epsilon}$ is nef, then we immediately stop and define $\rho$ to be the identity on $X$; 
in this case, properties (1)-(4) are straightforwardly satisfied.
Thus, we may and will assume that 
$K_{(X, \cal F, \Delta), \epsilon}$ 
is not nef; 
hence, there exists an extremal ray 
$R_0 \subset \ccurves{X}$ 
such that 
$K_{(X, \cal F, \Delta), \epsilon} \cdot R_0 < 0$.
Since, by definition, 
$K_{(X, \cal F, \Delta), \epsilon} = (K_{\cal F}+\Delta_{n-inv})+\epsilon (K_X+\Delta)$, 
then 
\begin{align*}
\text{either} 
\quad 
(K_{\cal F}+\Delta_{n-inv}) \cdot R_0< 0, 
\quad
\text{or} 
\quad 
(K_X+\Delta) \cdot R_0 < 0.
\end{align*}
We shall consider the two cases separately:
\begin{enumerate}
\item[(i)] 
$(K_{\cal F}+\Delta_{n-inv}) \cdot R_0<0$:
then $R_0$ is spanned by the class of an $\mathcal F$-invariant rational curve $C_0$, ~\cite[Theorem 6.3]{Spicer17};
moreover, Theorem~\ref{thm_recall_surf_mmp} shows that $R_0$ can be contracted.

\item[(ii)] $(K_X+\Delta) \cdot R_0< 0$:  
then $R_0$ can be contracted by the classical version of the Cone and Contraction Theorem, e.g.,~\cite[Theorem~3.7]{KM98}.
\end{enumerate}
We denote the contraction constructed in (i)/(ii) by $\rho_1\colon X \rightarrow X_1$.
If both 
$(K_{\cal F}+\Delta_{n-inv}) \cdot R_0< 0$ 
and 
$(K_X+\Delta) \cdot R_0 < 0$,
then the contraction $\rho_1$ is independent of the choice of (i)/(ii)
because the contracted curves coincide.

If $\dim X_1 < \dim X$, then we stop and take $\rho$ to be the identity again.
In this case then $K_{(X, \cal F, \Delta), \epsilon}$ is not pseudo-effective, since $R_0$ is contained in the cone of movable curves of $X$, as the fibers of $\rho_1$ move and their classes all lie in $R_0$.

If $\dim X_1 = \dim X$, then $\rho_1$ is birational and we set  
$\mathcal F_1:=\rho_{1\ast} \mathcal F$, 
$\Delta_1 := \rho_{1\ast} \Delta$.
Lemma~\ref{lem_ed_sing_pres} implies that $(X_1, \mathcal F_1, \Delta_1)$ is $(\epsilon, \delta)$-\alc.

\medskip

{\bf Claim 1}. 
{\it If $\rho_1$ is birational, then  $(X_1, \Delta_1)$ is dlt and $(\mathcal F_1, \Delta_{1, n-inv})$ is log canonical}.

\begin{proof}[Proof of the Claim 1]
We first deal with the singularities of $(X_1, \Delta_1)$.
If $\rho_1$ is obtained via (ii), then the conclusion follows at once from the negativity lemma since $(X, \Delta)$ is dlt and $\rho_1$ is a $(K_X+\Delta)$-negative contraction.
Hence, we can assume that $\rho_1$ is a $(K_{\cal F}+\Delta_{n-inv})$-negative birational contraction.
Denoting by $C_0$ the rational invariant curve contracted by $\rho_1$, then $\cal F$ is canonical in a neighborhood of $C_0$ by Lemma~\ref{lem_lc_mmp} and the conclusion then follows from~Remark~\ref{remark_dlt_preserved}, since $\rho_1$ is a $(K_X+\Delta+C_0)$-negative contraction and the pair $(X, \Delta+C_0)$ is dlt.

To show that $(\mathcal F_1, \Delta_{1, n-inv})$ is log canonical, it suffices to apply Lemma~\ref{cor_bound2}, since $(X_1, \mathcal F_1, \Delta_1)$ is $(\epsilon, \delta)$-\alc.
\end{proof}

We may then substitute $(X, \mathcal F, \Delta)$ with $(X_1, \mathcal F_1, \Delta_1)$ and repeat the above process.
We therefore obtain a sequence of contractions
\begin{align}
\label{eqn:mmp.seq}
\xymatrix{
X=: X_0 \ar[r]^{\rho_1} & 
X_1 \ar[r]^{\rho_2} &
\dots \ar[r]^{\rho_i} &
 X_i \ar[r]^{\rho_{i+1}} &
\dots .
}
\end{align}
We set, inductively, $\mathcal F_j := \rho_{j\ast} \mathcal F_{j-1}$ and $\Delta_j:=\rho_{j\ast} \Delta_{j-1}$.
At each step, $\rho_j$ contracts an extremal ray $R_{j-1} \subset \ccurves{X_{j-1}}$ having negative intersection with $K_{(X_{j-1}, \cal F_{j-1}, \Delta_{j-1}), \epsilon}$. 
Since $X$ is a projective surface, $X_j$ is projective as well for all $j$. 
Since the Picard number of $X_j$ decreases at each step,
the above sequence of contractions cannot be infinite.
Hence, there exists $n\in \mathbb N$ such that
\begin{enumerate}
\item[(a)]
either $K_{(X_n, \cal F_n, \Delta_n), \epsilon}$ is nef, or
\item[(b)]
there exists a $K_{(X_n, \cal F_n, \Delta_n), \epsilon}$-negative extremal ray $R_n \subset \ccurves{X_n}$ whose contraction induces a fibration $X_{n} \to X_{n+1}$ with $\dim X_{n}>\dim X_{n+1}$.
Thus, $K_{(X_n, \cal F_n, \Delta_n), \epsilon}$ is not pseudo-effective, since $R_n$ is spanned by curves that move in $X_n$. 
\end{enumerate}
In both cases (a), (b), we set $Y:=X_n$ and $\rho=\rho_{n-1} \circ \rho_{n-2} \circ \dots \circ \rho_2 \circ \rho_1$ and in case (b) we set $Z:=X_{n+1}$.

We are then ready to prove the 4 statements of the Theorem.
\begin{enumerate}
\item 
The conclusion follows inductively from Lemma~\ref{lem_ed_sing_pres}, since at each step of~\eqref{eqn:mmp.seq}, $\rho_j$ contracts the $K_{(X_{j-1}, \mathcal F_{j-1}, \Delta_{j-1}), \epsilon}$-negative extremal ray $R_{j-1}$.
\item 
As $(Y, \mathcal G, \Delta)$ is $(\epsilon, \delta)$-\alc by (1), Lemma~\ref{lem_perturb_epsilon}.\ref{lem_perturb_epsilon:part1} implies that 
it is also $\epsilon$-\alc;
thus, by our choice of $\epsilon$, Proposition \ref{cor_bound2} implies that
$(\cal G, \Gamma_{n-inv})$ 
has log canonical singularities.
The conclusion on the singularities of $(Y, \Gamma)$ follows applying inductively the same proof as that of Claim 1.
\item 
If $K_{(X, \cal F, \Delta), \epsilon}$ is pseudo-effective, then $K_{(X_j, \cal F_j, \Delta_j), \epsilon}$ is also pseudo-effective for all $j=1, \dots, n$,  since 
$K_{(X_j, \cal F_j, \Delta_j), \epsilon}$
is the pushforward of 
$K_{(X_{j-1}, \cal F_{j-1}, \Delta_{j-1}), \epsilon}$.
Hence, the sequence of contractions in~\eqref{eqn:mmp.seq} must conclude with case (a) above. 
\item 
If $K_{(X, \cal F, \Delta), \epsilon}$ is not pseudo-effective, then also $K_{(X_j, \cal F_j, \Delta_j), \epsilon}$ is not pseudo-effective for all $j=1, \dots, n$, by the negativity lemma.
Hence, the sequence of contractions in~\eqref{eqn:mmp.seq} must terminate with (b) above. By ~\cite[Corollary 3.17]{KM98} and 
Theorem~\ref{thm_recall_surf_mmp}, then $\rho(X_n/X_{n+1}) = 1$.
\end{enumerate}
\end{proof}

Exactly as in the classical case, the proof of Theorem~\ref{thm_adj_mmp} can be adapted to yield a proof of the following relative version of the statement.
The interested reader can find a detailed explanation of how to reduce from the MMP on a projective variety to the relative case in~\cite[\S 3.6-7]{KM98}.

\begin{corollary}
\label{cor_rel_mmp}
Let $(X, \cal F, \Delta)$ be a foliated triple where $X$ is a surface, $(X, \Delta)$ is dlt and  $\cal F$ is rank one and let $\pi\colon X \rightarrow S$ be a projective morphism.
Fix $\delta \geq 0$.
Suppose that $\Delta \in I \cap [0, 1-\delta]$ where $I\subset [0, 1]$ is a DCC set. 
Fix $0<\epsilon <E(I)$ where $E(I)$ is as in Proposition~\ref{cor_bound2}.
Suppose that $(X, \cal F, \Delta)$ is $(\epsilon, \delta)$-\alc.

Then we may run a 
$K_{(X, \cal F, \Delta), \epsilon}$-MMP relative over $S$
\begin{align*}
\xymatrix{
X \ar[rr]^{\rho} \ar[dr]& & Y \ar[dl]\\
& S &
}
\end{align*} 
Moreover,  setting $\cal G~:=~\rho_*\cal F$, $\Gamma:= \rho_\ast \Delta$ the following properties hold:

\begin{enumerate}
\item $(Y, \cal G, \Gamma)$ is $(\epsilon, \delta)$-\alc.

\item $(\cal G, \Gamma_{\text{n-inv}})$ has log canonical singularities and $(Y, \Gamma)$ is dlt.

\item If $K_{(X, \cal F, \Delta), \epsilon}$ is pseudo-effective over $S$ then
$K_{(Y, \cal G, \Gamma), \epsilon}$ is nef over $S$.

\item If $K_{(X, \cal F, \Delta), \epsilon}$ is not pseudo-effective over $S$ then $Y/S$ admits a fibration $f\colon Y \rightarrow Z$ over $S$ with $\rho(X/Y) = 1$ and such that $-K_{(Y, \cal G, \Gamma), \epsilon}$ is $f$-ample.
\end{enumerate}
\end{corollary}

When running the adjoint MMP we have precise control of the singularities of the underlying surface, as the following corollary shows.

\begin{corollary}
\label{cor_bounded_sing}
Fix a positive real number $\delta$ and a DCC set $I \subset [0, 1]$.
Fix a positive real number $\epsilon <E(I)$ where $E(I)$ is as in Proposition~\ref{cor_bound2}.

Let $(X, \cal F, \Delta)$ be a projective log smooth foliated triple where $X$ is a surface and $\mathcal F$ is rank one.
Assume that $\Delta \in I \cap [0, 1-\delta]$.
Let $\rho\colon (X, \cal F, \Delta) \rightarrow (Y, \cal G, \Theta \coloneqq \rho_*\Delta)$ be a (finite) sequence of steps of the $K_{(X, \cal F, \Delta), \epsilon}$-MMP.
Then $(Y, \Theta)$ has $\eta$-lc singularities where $\eta = \frac{\epsilon\delta}{1+\epsilon}$.
\end{corollary}

\begin{proof}
We argue by induction on the number $n$ of steps of the $K_{(X, \Delta, \mathcal F), \epsilon}$-MMP that compose $\rho$
\begin{align*}
\xymatrix{
X_0\coloneqq X \ar[r]^{\rho_1} &
X_1 \ar[r]^{\rho_2} &
\dots \ar[r]^{\rho_{n-1}} &
X_{n-1} \ar[r]^{\rho_n} &
X_n = Y,} 
\quad \rho= \rho_n \circ \dots \circ \rho_1.
\end{align*}
We denote inductively, 
$\mathcal F_i:=\rho_{i, \ast} \mathcal F_{i-1}, 
\mathcal F_0\coloneqq \mathcal F$,
$\Delta_i\coloneqq \rho_{i, \ast} \Delta_{i-1}, \Delta_0\coloneqq \Delta$.
Since $(X, \Delta, \mathcal F)$ is log smooth and $\Delta \in [0, 1-\delta]$ then $(X, \Delta)$ is $\delta$-lc. 
As for $\delta>0$, 
$\delta > \eta=\frac{\epsilon\delta}{1+\epsilon}$, 
the case $n=0$ is trivially settled.
Hence, we will assume that $n>0$.
Moreover, the above observation implies that it suffices to show that $(Y, \Theta)$ is $\eta$-lc at all $P \in Y$ at which $\rho^{-1}$ is not an isomorphism, i.e., 
$P \in \rho({\rm exc}(\rho))$.
We fix one such point and we distinguish two cases based on whether or not $P$ is a point at which $(\cal G, \Theta_{\text{n-inv}})$ is terminal.

\medskip

{\bf Case 1}. 
{\it $P$ is terminal for $(\cal G, \Theta_{\text{n-inv}})$.}
\\
Let $L$ be the unique germ of a leaf through $P$ and let $E$ be the union of the curves $E_i$ contracted by $\rho$ to $P$ with the reduced structure.
The $E_i$ are all $\cal F$-invariant since $\mathcal G$ is non-dicritical.
Hence, 
$\rho_*^{-1}\Theta_{\text{n-inv}} = 
\Delta_{\text{n-inv}}$.  
Moreover, near every point of 
$E \cap \text{sing}(\cal F)$ 
we know that $\cal F$ admits a holomorphic first integral.
Thus, by \cite[Lemma 2.16]{CS20},
$K_{X}+\rho_*^{-1}(\Theta_{\text{n-inv}}+L)+E$
is $\rho$-numerically equivalent to 
$K_{\cal F}+\Delta_{\text{n-inv}}$
over a sufficiently small neighborhood $U_P$ of $P \in Y$.
Moreover, 
$\Theta_{\text{n-inv}}+L \geq \Theta$, as $L$ is the unique leaf of $\mathcal F$ through $P$.
Writing 
$\Delta = \rho_*^{-1}\Theta+\sum a_iE_i$, 
then $a_i \leq 1-\delta$ by assumption and
$K_{(X, \cal F, \Delta), \epsilon}$
is $\rho$-numerically equivalent over $U_P$ to
\begin{align*}
&(K_{X}+\rho_*^{-1}(\Theta_{\text{n-inv}}+L)+E)+\epsilon(K_X+\rho_*^{-1}\Theta+\sum a_iE_i) = \\
& (1+\epsilon)(K_X+\frac{1}{1+\epsilon}(\rho_*^{-1}(\Theta_{\text{n-inv}}+L)+\epsilon \rho_*^{-1}\Theta)+\sum \frac{1+\epsilon a_i}{1+\epsilon} E_i).
\end{align*}
Since for all $i$, 
$(\rho_*^{-1}(\Theta_{\text{n-inv}}+L)+\epsilon \rho_*^{-1}\Theta) \cdot E_i 
\geq 
(1+\epsilon)(\rho_*^{-1}\Theta) \cdot E_i$ and the support of $L$ is not contained in $E$,
it follows that over $U_P$ $\rho$ coincides with a run of the $(K_X+\rho_*^{-1}\Theta+\sum \frac{1+\epsilon a_i}{1+\epsilon} E_i)$-MMP.
As $(X, \rho_\ast^{-1} \Theta + \sum E_i)$ is log smooth by assumption,
$(X, \rho_\ast^{-1}\Theta)$
is $\delta$-lc by assumption, and 
$\frac{1+\epsilon a_i}{1+\epsilon} \leq \frac{1+\epsilon (1-\delta)}{1+\epsilon}$, 
then 
\begin{align*}
K_X+\rho_*^{-1}\Theta+\sum \frac{1+\epsilon a_i}{1+\epsilon} E_i
\leq 
\eta
(K_X+\rho_*^{-1}\Theta) +
(1-\eta)
(K_X+\rho_*^{-1}\Theta+\sum E_i),
\end{align*}
for 
$\eta \coloneqq \frac{\epsilon\delta}{1+\epsilon}$.
Thus, the pair $(Y, \Theta)$ is $\eta$-lc since it is the outcome of a run of the MMP over $U_P$ for the $\eta$-lc pair 
$(X, \rho_*^{-1}\Theta+\sum \frac{1+\epsilon a_i}{1+\epsilon} E_i)$. 
\medskip


{\bf Case 2}. 
{\it $P$ is not terminal for $(\cal G, \Theta_{\text{n-inv}})$.}  
\\ 
By the inductive hypothesis, we can assume that $(X_{n-1}, \Delta_{n-1})$ is $\eta$-lc.
Moreover, we can assume that $\rho_n \colon X_{n-1} \to Y$ is the contraction of a $K_{(X_{n-1}, \mathcal F_{n-1}, \Delta_{n-1}), \epsilon}$-negative rational curve $C_{n-1}$ with $\rho_n(C_{n-1})=P$.
If $\rho_n$ is a $(K_{X_{n-1}}+ \Delta_{n-1})$-negative contraction, then the proof terminates.
Hence, we can assume that $\rho_n$ is a $(K_{\mathcal F_{n-1}}+ \Delta_{n-1})$-negative contraction.
This leads to an immediate following contradiction, by the following claim, which concludes the proof.

\medskip

{\bf Claim}. 
{\it If 
$\rho_n$ is a $(K_{\mathcal F_{n-1}}+ \Delta_{n-1})$-negative contraction,
then 
$(\mathcal G, \Theta_{n-inv})$ 
is terminal at $P$.
}
\begin{proof}[Proof of the Claim]
By Lemma~\ref{lem_lc_mmp},  
$(\mathcal F_{n-1}, \Delta_{n-1})$ 
is canonical at any point 
$Q \in X_{n-1}$ such that $\rho_n(Q)=P$, since $\rho_{n}$ is birational.
The negativity lemma readily shows that the terminality of $(\mathcal G, \Theta_{n-inv})$ at $P$ follows from the hypothesis of the claim, cf.~\cite[Lemma~3.38]{KM98}.
\end{proof}
\end{proof}

We can also show that $(\epsilon, \delta)$-adjoint log canonical models exist in the projective category.

\begin{corollary}
\label{cor_can_model}
Set up as in Theorem \ref{thm_adj_mmp}. 
Suppose in addition that $(X, \Delta)$ is klt and that $K_{(X, \cal F, \Delta), \epsilon}$ is big. 
Let $(Y, \cal G, \Gamma)$ be the output of a run of the $K_{(X, \cal F, \Delta), \epsilon}$-MMP starting on $X$.
Then, there exists a birational contraction 
$p\colon (Y, \cal G, \Gamma) \rightarrow (Y_{can}, \cal G_{can}, \Gamma_{can})$
such that
\begin{enumerate}
\item $Y_{can}$ is projective;

\item $K_{(Y_{can}, \cal G_{can}, \Gamma_{can}), \epsilon}$ is an ample $\bb Q$-Cartier divisor;

\item $(Y_{can}, \cal G_{can}, \Gamma_{can})$ has $(\epsilon, \delta)$-\alc singularities;

\item $Y_{can}$ has $\eta$-lc singularities where $\eta = \frac{\epsilon\delta}{1+\epsilon}$.
\end{enumerate}
Moreover, $\displaystyle Y_{can}$ is uniquely determined.
\end{corollary}

\begin{proof}
First, note that $(Y, \cal G, \Gamma)$ has $(\epsilon, \delta)$-\alc singularities.

Let $C \subset Y$ be a curve with $K_{(Y, \cal G, \Gamma), \epsilon}\cdot C = 0$.  
Since $K_{(Y, \cal G, \Gamma), \epsilon}$ is nef and big by construction, then $C^2<0$ by the Hodge Index Theorem.
There are three possibilities at this point: 
\begin{enumerate}
\item[(i)] $(K_{\cal G}+\Gamma_{\text{n-inv}})\cdot C<0$, 
\item[(ii)] $(K_Y+\Gamma)\cdot C<0$, or 
\item[(iii)] $(K_{\cal G}+\Gamma_{\text{n-inv}})\cdot C = (K_Y+\Gamma)\cdot C = 0$.
\end{enumerate}

In case (i), 
$(\cal G, \Gamma_{\text{n-inv}})$ 
has canonical singularities in a neighborhood of 
$C$, 
as otherwise 
$C$ 
would move, c.f., the proof of  Theorem~\ref{thm_adj_mmp}; 
thus, we may contract 
$C$ 
by a 
$(K_{\cal G}+\Gamma_{\text{n-inv}})$-negative 
contraction.
In case (ii) we may contract 
$C$ 
by a 
$(K_Y+\Gamma)$-negative 
contraction.
In case (iii) we perform a 
$(K_Y+\Gamma+tC)$-negative contraction for some $t>0$ sufficiently small, which does not constitute a problem, since we are assuming that $(X, \Delta)$ is klt to start with.
In any case, we obtain a morphism $p'\colon Y \rightarrow Y'$ which contracts $C$ to a point.
$(Y', p'_{*}\cal G, p'_*\Gamma)$ is still $(\epsilon, \delta)$-\alc and the argument in 
Corollary \ref{cor_bounded_sing}
works equally well here to show 
that $Y'$ has $\eta$-lc singularities.

Repeating this process we will eventually terminate in model $(Y_0, \cal G_0, \Gamma_0)$ such that
 $K_{(Y_0, \cal G_0, \Gamma_0), \epsilon} \cdot C>0$ for all 
curves $C$, and hence $K_{(Y_0, \cal G_0, \Gamma_0), \epsilon}$ is ample.
That immediately implies the final part of the statement of the corollary, as $(p \circ \rho)^\ast K_{(Y_0, \cal G_0, \Gamma_0), \epsilon}$ realizes the positive part of the Zariski decomposition of $K_{(X, \mathcal F, \Delta), \epsilon}$ which is uniquely determined.
\end{proof}

\begin{remark}
\label{rmk_rel_ample_model}
Analogously to Corollary  \ref{cor_rel_mmp}, also Corollary \ref{cor_can_model} has a relative version with respect to a projective morphism $\pi \colon X \to S$.

Following the notation of Corollary~\ref{cor_rel_mmp}, we assume in addition that 
$(X, \Delta)$ 
is klt and that 
$K_{(X, \cal F, \Delta), \epsilon}$ 
is big over $S$. 
Denoting $(Y, \cal G, \Gamma)$ be the output of a run of the $K_{(X, \cal F, \Delta), \epsilon}$-MMP over $S$, starting from $X$, there exists a birational contraction 
$p\colon (Y, \cal G, \Gamma) \rightarrow (Y_{can}, \cal G_{can}, \Gamma_{can})$
such that
\begin{enumerate}
\item $Y_{can}$ is projective over $S$;

\item $K_{(Y_{can}, \cal G_{can}, \Gamma_{can}), \epsilon}$ is an ample $\bb Q$-Cartier divisor over $S$;

\item $(Y_{can}, \cal G_{can}, \Gamma_{can})$ has $(\epsilon, \delta)$-\alc singularities;

\item $Y_{can}$ has $\eta$-lc singularities where $\eta = \frac{\epsilon\delta}{1+\epsilon}$.
\end{enumerate}
Moreover, $\displaystyle Y_{can}/S$ is uniquely determined.
\end{remark}

The last corollary motivates the following definition of an ample model.

\begin{defn}
We call $(Y_{can}, \cal G_{can}, \Gamma_{can})$ the {\bf $(\epsilon, \delta)$-\alc model} of $(X, \cal F,  \Delta)$ (or the 
{\bf $\epsilon$-\acanonical model} when $\Delta = 0$).  
\end{defn}

Corollary~\ref{cor_can_model} marks a notable difference with the theory of the MMP for the canonical divisor of general type foliations with canonical singularities.
We recall that the canonical model of a surface foliation, in the sense of \cite{McQuillan08},
is not necessarily projective, owing to the presence of cusp type singularities arising from the contraction of elliptic Gorenstein leaves (e.g.l.s) to points.  
We emphasize that the
$\epsilon$-\acanonical model, by contrast, is always projective and does not contract any e.g.l.s to points.

\section{A bound on the pseudo-effective threshold in the big case}
The goal of this section is to prove the following.

\begin{theorem}
\label{thm_big_psef}
Fix a DCC set $I \subset [0,1]$. 
There exists a positive real number $\tau=\tau(I)$ satisfying the following property:

If $(X, \cal F, \Delta)$ is a projective log smooth triple such that $X$ is a surface, $\mathcal F$ is rank one, 
$K_{\cal F}+\Delta_{n-inv}$ is big, and $\Delta \in I$, then $K_{(X, \cal F, \Delta), \tau}$ is big.
\end{theorem}

\subsection{Outline of the proof of Theorem \ref{thm_big_psef}}

We divide the outline of the proof into steps.

\medskip

{\bf Step 0} As a preliminary step we show that we may freely assume that $\Delta$ has coefficients
in some fixed finite subset $J \subset I$ (Proposition \ref{prop_shrink_boundary}).  

For ease of exposition
we will assume that $\Delta = 0$ for the remainder of the proof sketch.

\medskip

{\bf Step 1} Fix some $0<\tau < E(I)$, where $E(I)$ is the constant given by Lemma~\ref{cor_bound2}.
If $K_{(X, \cal F), \tau}$ is pseudo-effective then we are done.  
Thus, we may assume that $K_{(X, \cal F), \tau}$ is not pseudo-effective.

\medskip

{\bf Step 2} 
Assume for the moment that there exists a bounded family $\ff M$
such that for every $(X, \cal F)$ with $K_{(X, \cal F), \tau}$ not pseudo-effective
there exists $(Z, \cal G) \in \ff M$ and a birational morphism $(X, \cal F) \rightarrow (Z, \cal G)$.
Assume moreover that $Z$ is smooth at all the strictly log canonical singularities of $\cal G$.

By boundedness we may then find a $\tau'$ (independent of $(Z, \cal G)$)
such that $K_{(Z, \cal G), \tau'}$ is pseudo-effective.   

In general $(Z, \cal G)$ may not be $\tau'$-\acanonical and so we cannot lift sections
to $(X, \cal F)$.  However, if $(Z, \cal G)$ is not $\tau'$-\acanonical
using some computations from \cite{PS16},  see Lemma \ref{lem_fin_many_eigs},
we show that the eigenvalues of the singular points of $\cal G$ belong to a finite set.
In particular, it follows that we may resolve the log canonical singularities of $\cal G$ in a bounded way.
Repeating the argument and using notherian induction allows us to conclude.
This is achieved in Proposition~\ref{prop_bound_for_nice_family}.

\medskip

{\bf Step 3}
We may run an adjoint MMP which terminates in a foliated pair $(Z, \cal G)$ where
$Z$ is a Fano surface with $\eta$-lc singularities, where $\eta>0$ depends only on $\tau$ and $J$.
In particular, $(Z, \cal G)$ belongs to a bounded family.  If we can arrange it so that
$Z$ is smooth at all strictly log canonical singularities of $\cal G$ then we are done by
Step 2, this is done in Proposition \ref{prop_boundary_in_finite_set}.

\medskip

{\bf Step 4}
Finally, we show that we may modify our family $(Z, \cal G)$ so that the condition
that $Z$ is smooth at all strictly log canonical singularities of $\cal G$ holds (this is done
in Lemma \ref{lem_resolution_logsmooth_near_lc}).

\subsection{Proof of Theorem \ref{thm_big_psef}}

\begin{lemma}
\label{lem_fin_many_eigs}
Fix positive real numbers $\epsilon', \delta$.
Then there exists a finite set $\Lambda \subset \bb N \times \bb N$ 
depending only on $\epsilon'$
such that the following holds:

Let $(0 \in \bb C^2, \cal F)$ be a germ of a rank one foliation such that $\cal F$ has log canonical singularities, but is not $(\epsilon', \delta)$-\alc. 
Then $T_{\cal F}$ is generated by a vector field of the form 
$px\frac{\partial}{\partial x}+qy\frac{\partial}{\partial y}$  where $(p, q) \in \Lambda$.
\end{lemma}

\begin{proof}
Observe that $\cal F$ is strictly log canonical at $0$ and so 
$T_{\cal F}$ is generated by a vector field of the form 
$px\frac{\partial}{\partial x}+qy\frac{\partial}{\partial y}$ where $p, q$ are positive co-prime integers.
Let $[u_1, ..., u_k]$ be the continued fraction representation of $p/q$.

Next, observe that $\cal F$ is not $\epsilon'$-\acanonical.
It follows by combining \cite[Corollary 4.10]{PS16} - keeping in mind the slight difference of notations, Remark~\ref{rmk_different_notation} -
and~\cite[Lemma 4.7]{PS16} that 
\begin{align*}
\frac{1}{\epsilon'} \geq \sum_{i=1}^k u_i,
\end{align*}
which implies that $k$ and $u_i$ are bounded in terms of $\epsilon'$.
Hence, $p/q$ can take on only finitely many values, depending only on $\epsilon'$.
\end{proof}

\begin{proposition}
\label{prop_bound_for_nice_family}
Fix a finite subset $J \subset [0, 1)$ and a real number $\epsilon>0$.
Let $\mathfrak D$
be a collection of projective foliated log smooth triples
$(X, \cal F, \Delta)$
such that 
$X$ 
is a surface, 
$\mathcal F$ 
is rank one,
$\Delta = \Delta_{n-inv}$, $\Delta \in J$,
and
$K_{\cal F}+\Delta_{n-inv}$
is big.
Suppose that there exists a bounded family of 2-dimensional foliated triples 
\begin{align}
\label{eqn:bounded.family}
f \colon (\mathcal Y, \cal G, \Gamma) \rightarrow T
\end{align}
such that for all 
$(X, \cal F, \Delta) \in \mathfrak D$
there exists 
$t \in  T$
such that
\begin{enumerate}
\item 
there exists a contraction 
$\pi_t \colon X \rightarrow \mathcal Y_t$ 
and 
$\cal G_t= \pi_{t, \ast}\cal F$, 
$\Gamma_t= \pi_{t, \ast}\Delta$;

\item 
$K_{(X, \cal F, \Delta), \epsilon} = 
\pi_t^\ast K_{(\mathcal Y_t, \cal G_t, \Gamma_t), \epsilon}+E_{\epsilon}$, 
where 
$E_{\epsilon} \geq 0$;

\item 
$\mathcal Y_t$ 
is klt, 
$(\cal G_t, \Gamma_{t})$ 
is log canonical and 
$K_{\cal G_t}+\Gamma_{t}$ 
is big;

\item 
$(\mathcal Y_t, \Gamma_t)$ 
is log smooth at all strictly log canonical points of 
$(\cal G_t, \Gamma_{t})$.
\end{enumerate}

Then there exists a positive real number 
$\tau_0 =\tau_0(J, \epsilon)$
such that for all 
$0 \leq \tau'<\tau_0$ 
and all 
$(X, \cal F, \Delta)\in \mathfrak D$, 
$K_{(X, \cal F, \Delta), \tau'}$ 
is big.
\end{proposition}

\begin{proof}
By boundedness of the family in~\eqref{eqn:bounded.family} and by assumption (3), there exists 
$\tilde{\tau}_0 \leq \epsilon$ 
such that for all 
$0 \leq s \leq \tilde{\tau}_0$ 
and all 
$t \in T$, 
$K_{(\mathcal Y_t, \cal G_t, \Gamma_t), s}$ is 
big.
For a foliated triple
$(X, \cal F, \Delta)\in \mathfrak D$, 
a point
$t \in T$ 
as in the assumption of the proposition, and 
$\lambda \in \mathbb R_{\geq 0}$, 
we define the effective divisors 
$E_{\lambda, t}, F_{\lambda, t}$ 
by
\begin{align*}
E_{\lambda, t}-F_{\lambda, t} \coloneqq 
K_{(X, \cal F, \Delta), \lambda} - 
\pi^*K_{(\mathcal Y_t, \cal G_t, \Gamma_t), \lambda}
\end{align*}
where  we assume that no prime divisor on 
$X$ 
appears in the support of both 
$E_{\lambda, t}, F_{\lambda, t}$.
Assumption (2) implies that 
$F_{\epsilon, t} = 0$; 
moreover, if 
$F_{\tilde{\tau}_0, t} = 0$, 
then 
$K_{(X, \cal F, \Delta), \tilde{\tau}_0}$ 
is big.

\medskip

{\bf Claim 1}.
{\it
Fix 
$(X, \cal F, \Delta) \in \mathfrak D$ 
and 
$t \in T$ 
as in the statement of the proposition. 
Assume that $F_{\tilde{\tau}_0, t} \neq 0$.
Then there exists a strictly log canonical singularity of 
$\mathcal G_t$ 
at which 
$(\mathcal Y_t, \mathcal G_t, \Gamma_t)$ 
is not 
$\tilde{\tau}_0$-\acanonical.
}
\begin{proof}
By item~(2) and since 
$F_{\tilde{\tau}_0, t} \neq 0$, 
there must exist a prime divisor $C \subset X$ that is $\pi$-exceptional and such that for 
$a_C \coloneqq \mu_C \Delta$,
\begin{align*}
a(C, \cal G_t, \Gamma_t)+
\epsilon a(C, \mathcal Y_t, \Gamma_t) &\geq 
-(\iota(C)+\epsilon)a_C,
\\
a(C, \cal G_t, \Gamma_t)+
\tilde{\tau}_0 a(C, \mathcal Y_t, \Gamma_t) &< 
-(\iota(C)+\tilde{\tau}_0)a_C.
\end{align*}
Hence, 
$a(C, \cal G_t, \Gamma_t) <0$ 
and
$\iota(C) = 1$, 
$C$ 
is a strictly log canonical place of 
$(\cal G_t, \Gamma_t)$, 
with center $P$ on 
$\mathcal Y_t$, 
and 
$a(C, \cal G_t, \Gamma_t) = -1$, 
as otherwise 
$\cal G_t$ 
would be non-dicritical at 
$P$ 
contradicting that
$a(C, \cal G_t, \Gamma_t) <0$.
Moreover, 
$P \notin \text{supp}(\Gamma_t)$
and 
$\cal G_t$ is not 
$\tilde{\tau}_0$-\acanonical 
at $P$.
\end{proof}

Consider the following subset $Z_0 \subset T$
\begin{align}
\nonumber
Z_0 \coloneqq
\{
t' \in T \
\vert \
& \text{there exists a strictly log canonical centre 
$P' \in \mathcal Y_{t'}$}\\
\label{def:Z_0}
& \text{for 
$\cal G_{t'}$
at which 
$(\mathcal Y_{t'}, \cal G_{t'}, \Gamma_{t'})$
is not 
$\tilde{\tau}_0$-\acanonical.}
\}
\end{align}
By Lemma \ref{lem_fin_many_eigs} and item (4) in the hypotheses of the proposition, the eigenvalues of 
$\cal G_{{t'}}$ 
at a point 
$P' \in \mathcal Y_{t'}$ 
as in~\eqref{def:Z_0} belong to a finite set;
thus,
$Z_0 \subset T$ 
is a Zariski closed subset -- with possible equality.

Given 
$(X, \mathcal F, \Delta) \in \mathfrak D$ 
such that for the foliated triple 
$(\mathcal Y_t, \cal G_t, \Gamma_t)$ 
associated to it, 
$t \in T \setminus Z_0$, 
then 
$F_{\tilde{\tau_0}, t}=0$ 
and for all 
$0 \leq s \leq \tilde{\tau}_0$,
$K_{(X, \cal F, \Delta), s}$ 
is big.

Up to stratifying $Z_0$ into a finite disjoint union of locally closed set, we may and will assume that 
$Z_0$ 
is the union of finitely many (disjoint) smooth irreducible components.
We denote by 
$\overline f \colon
(\overline{\mathcal Y}, 
\overline{\mathcal G}, 
\overline{\Gamma}) 
\to Z_0$ 
the restriction of 
$(\cal Y, \cal G, \Gamma)$
to $Z_0$.

\medskip

{\bf Claim 2}.
{\it 
There exists a surjective morphism 
$e \colon Z_0' \to Z_0$ 
finite onto its image 
and a bounded family of foliated triples 
\begin{align*}
f_0 \colon (\mathcal Y^0, \mathcal G^0, \Gamma^0) \to Z'_0
\end{align*}
such that 
\begin{enumerate}[label=(\roman*)]
\item for any $t' \in Z'_0$ there exists a morphism 
$\mu_{t'} \colon \mathcal Y^0_{t'} \rightarrow \mathcal Y_{e(t')}$ 
which is a foliated log resolution in a neighborhood of any point of 
$\mathcal  Y_{e(t')}$ 
at which  
$(\mathcal Y_{e(t')}, \Gamma_{e(t')})$
is not 
$\tilde{\tau_0}$-adjoint canonical;

\item if 
$(X, \mathcal F, \Delta)$ 
corresponds to 
$(\mathcal Y_t, \mathcal G_t, \Gamma_t)$, $t \in Z_0$ 
as in the statement of the proposition, 
then we may assume there exists
$t' \in e^{-1}(t)$ and
a morphism $\nu_{t'}\colon X \rightarrow \mathcal Y^0_{t'}$;

\item 
with the notation of item (ii),
then 
$\Gamma^0_{t'} = \nu_{t' \ast}\Delta$,
$(\mathcal Y^0_{t'}, \cal G^0_{t'}, \Gamma^0_{t'})$ 
is 
$\tilde{\tau}_0$-\acanonical 
at all log canonical centres of 
$(\cal G^0_{t'}, \Gamma^0_{t'})$, 
and
$K_{\cal G^0_{t'}}+\Gamma^0_{t'}$ is big.
\end{enumerate}

Moreover,
\begin{align}
\label{eqn:discrep.4.3}
K_{(X, \cal F, \Delta), \epsilon} = 
\nu_t^\ast K_{(\mathcal Y^0_t, \cal G^0_t, \Gamma^0_t), \epsilon}+\tilde{E}_{t, \epsilon}, 
\end{align}
with $\tilde{E}_{t, \epsilon} \geq 0$.
}

\begin{proof}
Passing to a stratification into locally closed subsets and a finite cover of $Z_0$, 
we may assume that the Zariski closed set $S'$ defined by
\begin{align*}
S' \coloneqq
\{
s \in \overline{\mathcal Y} \
\vert \ & 
\text{$(\overline{\cal G}_{\overline f(s)}, \overline{\Gamma}_{\overline f(s)})$
is strictly log canonical at 
$s \in \overline{\mathcal{Y}}_{\overline f(s)}$}\\
& \text{and 
$(\overline{\mathcal Y}_{\overline{f}(s)}, \overline{\cal G}_{\overline f(s)}, \overline{\Gamma}_{\overline f(s)})$ 
is not 
$\tilde{\tau}_0$-\acanonical 
at $s$}
\}
\end{align*}
is flat over 
$Z_0$, 
all fibers of 
$\overline f \vert_{S'} \colon S' \to Z_0$ 
are everywhere reduced, and the ratio of the eigenvalues of 
$\overline{\cal G}_{\overline f(s)}$ 
at 
$s$ 
is constant on the components of 
$S'$.
The last claim is a simple consequence of Lemma~\ref{lem_fin_many_eigs}.
Analogous reasoning shows also that there exists an upper bound on the number of strictly log canonical singularities of 
$(\overline{\cal G}_t, \overline{\Gamma}_{t})$
independent of 
$t$.

Each of the log canonical foliated surface singularities parametrized by 
$S'$ 
admits a foliated log resolution by a bounded number of blow ups, 
and the bound on the number of blow-up depends only on
$\tilde{\tau}_0$, 
as shown in Lemma~\ref{lem_fin_many_eigs} and its proof, since such singularities are not $\tilde{\tau}_0$-\acanonical.
Moreover, as the ratio of the eigenvalues of $\overline{\cal G}_{\overline f(s)}$ at $s$ is constant on the components of $S'$, we can perform these blow-ups in family and, thus, obtain a bounded family of 2-dimensional triples 
\begin{align*}
\xymatrix{
(\cal Y', \cal G', (\mu')^{-1}_*\overline \Gamma) 
\ar[dr]^{f'} \ar[rr]^{\mu} 
& &
(\overline{\cal Y}, \overline{\cal G}, \overline \Gamma)
\ar[dl]^{\overline f}
\\
& Z_0 &
}
\end{align*}
where 
$\mu$ 
is the partial resolution whose construction we just explained and 
$\cal G' \coloneqq \mu^{-1}\cal G$.
Moreover, possibly passing to a stratification of 
$Z_0$ 
into locally closed sets, and a finite covering of the irreducible components of the stratification, we may assume that for any 
$t \in Z_0$ 
there exists a 1-1 correspondence between the irreducible components of the exceptional locus of 
$\mu$ 
and those of 
$\mu_t$,
over the irreducible component of 
$Z_0$ 
containing 
$t$;
moreover, we can assume that if 
$E'$ 
is a 
$\mu$-exceptional prime divisor, then
$\iota(E'_t)=\iota(E')$, 
for any $t$ contained in the image of 
$E'$.
Let us denote by $\{E_1, \dots, E_r\}$ the $\mu$-exceptional divisors that are not $\cal G'$-invariant.
We define $W\coloneqq J \cup \{0\}$.
We define 
$Z'_{0} \coloneqq \bigsqcup_{(a_1, \dots, a_r) \in W^r} Z_0$ 
and 
$e \colon Z'_0 \to Z_0$ 
to be the identity on each copy of $Z_0$ contained in $Z'_0$.
Then, we define 
\begin{align*}
(\mathcal Y^0, \mathcal G^0, \Gamma^0)
\coloneqq
\bigsqcup_{(a_1, \dots, a_r) \in W^r} 
(\cal Y', \cal G', (\mu')_\ast^{-1} \overline \Gamma + \sum_{i=1}^r a_i E_i).
\end{align*}
The morphism $f'$ induces a morphism 
$f_0 \colon (\mathcal Y^0, \mathcal G^0, \Gamma^0) \to Z_{0}'$, 
which yields a bounded family of $2$-dimensional foliated triples.

To prove items (ii)-(iii), let 
$r\colon X' \rightarrow X$ 
resolve the the indeterminacies of the rational map 
$X \dashrightarrow \mathcal Y^0_t$,
so that 
$(X', \mathcal F'\coloneqq r^{-1}\mathcal F, \Delta' \coloneqq r_*^{-1}\Delta)$ 
is foliated log smooth. 
We denote by $\nu'_t \colon X' \to Y^0_t$ the induced morphism. 
If $K_{(X', \mathcal F', \Delta'), t}$ is big, then $K_{(X,\mathcal F, \Delta), t}$ is big.
Moreover for any 
$t>0$ 
we may write 
$K_{(X', \mathcal F', \Delta'), t} = r^*K_{(X, \mathcal F, \Delta), t}+G_{t}$ where $G_{t} \geq 0$;
thus, we are free to replace $(X, \mathcal F, \Delta)$ by $(X', \mathcal F', \Delta')$ 
in the statement of the Proposition.  
By construction, 
$\pi_{t\ast} \Delta= \Gamma_t$.
Since, 
$\Delta \in J$ 
and 
$\Delta'=r_\ast^{-1}\Delta$,
then
$(\nu'_{t\ast} \Delta' - \mu_{t\ast}^{-1} \Gamma_t)\in W$ and its support is contained only on  
$\mu_t$-exceptional 
components that are not 
$\cal G'$-invariant, 
by the definition of triples in 
$\mathfrak D$;
thus, there exists 
$(a_1, \dots, a_r) \in W^r$ 
such that 
$\nu'_{t\ast} \Delta' - (\mu')^{-1}_{t\ast} \Gamma_t = \sum_i a_i E_{i, t}$.
This completes the proofs of items (ii)-(iii).

We now prove~\eqref{eqn:discrep.4.3}.
Away from $\nu_t^{-1}(\text{exc}(\mu_t))$ this is clear;
thus, let 
$\Sigma$ be a connected component of 
$\text{exc}(\mu_t)$ 
and let 
$P \coloneqq \mu_t(\Sigma)$.
Since $P$ is a strictly log canonical singularity of 
$\cal G_t$, 
$P$ 
is not contained in the support of 
$\Gamma_{t}$.  
Moreover by Lemma \ref{lem_uniq_lc_place} there is exactly one divisor contained 
$\Sigma$ 
which is transverse to 
$\cal G^0_t$.
Hence, in a neighborhood of 
$\Sigma$, 
$\Gamma^0_{t}$ is supported on at most 1 curve.
Since $(Y^0_t, \cal G^0_t, \Gamma^0_t)$ is foliated log smooth in a neighborhood of $\Sigma$ it follows
that $(Y^0_t, \cal G^0_t, \Gamma^0_t)$ is in fact $\epsilon$-\acanonical for all $\epsilon>0$.
\end{proof}

\medskip

By boundedness of 
$Z'_0$, 
there exists 
$\tilde{\tau}_1 <\tilde{\tau}_0$ 
such that for all 
$t \in Z'_0$,
$K_{(Y_t^0, \cal G^0_t, \Gamma^0_t), \tilde{\tau}_1}$ 
is pseudo-effective.
Let 
$Z'_1 \subset Z'_0$ 
be the Zariski closed subset
\begin{align*}
Z_1 \coloneqq
\{
t' \in Z'_0 \
\vert \
& \text{there exists a strictly log canonical centre 
$P' \in \mathcal Y^0_{t'}$}\\
& \text{for 
$\cal G^0_{t'}$
at which 
$(\mathcal Y^0_{t'}, \cal G^0_{t'}, \Gamma^0_{t'})$
is not 
$\tilde{\tau}_1$-\acanonical.}
\}
\end{align*} 

We may then repeat the above argument with 
$Z'_1$ 
and we define $Z_1 \coloneqq e(Z_1)$.  
Iterating this process we produce a decreasing sequence of Zariski closed subsets
$Z_i \subsetneq Z_{i-1}$ 
of 
$T$ 
and a decreasing sequence of positive real numbers 
$0<\tilde{\tau}_i < \tilde{\tau}_{i-1}$
such that if
$(X, \cal F, \Delta) \in \mathfrak D$ 
and the corresponding point $t \in T$ given by the proposition satisfies
$t \in Z_{i-1}\setminus Z_i$,
then
$K_{(X, \cal F, \Delta), \tilde{\tau}_i}$ 
is big;
moreover, the foliated surface triples parametrized by points of 
$Z_{i}$ 
admit a log canonical singularity
which is not 
$\tilde{\tau}_{i-1}$-\acanonical.

This process must eventually terminate since at each step of the process, we reduce the number of strictly log canonical singularities on a foliated surface triple appearing in the fibers of our family.
Hence, we must eventually obtain that for some 
$n \gg 0$, 
$Z_n = \emptyset$ 
and
$K_{(X, \cal F, \Delta), \tilde{\tau}_{n-1}}$
is big for all 
$(X, \cal F, \Delta) \in \mathfrak D$.
Hence, we set
$\tau_0 \coloneqq \tilde{\tau}_{n-1}$.
\end{proof}

\begin{lemma}
\label{lem_resolution_logsmooth_near_lc}
Let
$h \colon (\cal Z, \cal L, \Xi) \to T$
be a bounded family of 2-dimensional projective foliated triples $(Z_t, \cal L_t, \Xi_t)$.
Assume that for all 
$t \in T$, 
$\Xi_{t, n-inv} = \Xi_t$ and
$(\cal L_t, \Xi_t)$ is log canonical.

Passing to a stratification of 
$T$ 
into locally closed sets, and a finite covering of the irreducible components of the stratification, there exists a bounded family of 2-dimensional projective foliated triples 
$j \colon (\cal Y, \cal G, \Gamma) \to T$
and a birational morphism over $T$,
$g \colon \cal Y \to \cal Z$,
such that for all $t \in T$,
\begin{enumerate}
\item 
$\cal G_t \coloneqq g^{-1}\cal L_t$ 
and 
$\Gamma_t \coloneqq g_\ast^{-1}\Xi_t$;

\item 
$(Y_t, \Gamma_t)$ 
is log smooth in a neighborhood of 
$g_t^{-1}(P)$ 
where $P$ is a strictly log canonical point of 
$(\cal L_t, \Xi_t)$;

\item 
$g_t$ 
only extracts divisors of discrepancy (resp. foliation discrepancy) 
$\leq 0$ 
(resp. 
$=-\iota(E)$);

\item 
any foliated log resolution 
$\tau_t \colon \overline{Z}_t \rightarrow Z_t$ 
of 
$(Z_t, \cal L_t, \Xi_t)$
factors as 
\begin{align*}
\xymatrix{
\overline{Z}_t \ar[r] \ar@/_/[rr]_{\tau_t}& Y_t \ar[r]^{g_t} & Z_t.
}
\end{align*}
\end{enumerate}

\end{lemma}
\begin{proof}
Fix 
$t \in T$ 
and let 
$P \in \cal Z_t$ 
be a point where 
$(\cal L_t, \Xi_t)$ 
is strictly log canonical.  
Thus, 
$P \notin \text{supp}(\Xi_t)$ 
and
$\cal L_t$
is strictly log canonical at $P$.
By Lemmata~\ref{lem_resolution} and~\ref{lem_strictly_lc_goren}, there exists a resolution 
$g_t\colon Y \rightarrow \cal Z_t$
by blowing up 
$\cal Z_t$
in 
$\cal L_t$-invariant 
centres.  
These blow ups only extract divisors of foliation discrepancy 
$=-\iota(E)$; 
taking 
$g_t$
to be a minimal log resolution of 
$Z_t$ 
around 
$P$, 
$g_t$
only extracts divisors of discrepancy $\leq 0$.  
Thus items (2)-(4) are satisfied.

Substituting 
$T$ 
with a stratification and taking finite covers of components, the minimal resolutions 
$g_t \colon Y_t \to \cal Z_t$
fit together in family to form a bounded family
$j \colon \cal Y \to T$
of resolutions
$g \colon \cal Y \to Z$ 
such that 
$g\vert_{\cal Y_t}=g_t$.
To conclude, it suffices to define 
$\cal G \coloneqq g^{-1} \cal L$, 
$\Gamma \coloneqq g_\ast^{-1}\Xi$.
\end{proof}

\begin{proposition}
\label{prop_boundary_in_finite_set}
Let $J \subset [0, 1)$ be a finite subset.
Let $\mathfrak D_J$ be the set of all triples  $(X, \cal F, \Delta)$ such that
\begin{enumerate}
\item 
$(X, \cal F, \Delta)$ is a projective foliated log smooth triple, $X$ is a surface, $\mathcal F$ is rank one

\item 
$\Delta_{n-inv} = \Delta \in J$, and

\item
$K_{\cal F}+\Delta_{n-inv}$ is big.
\end{enumerate}
Then there exists $\tau_0=\tau_0(J)>0$ such
that for all $0\leq \epsilon<\tau$ and for any triple $(X, \cal F, \Delta) \in \mathfrak D_J$, $K_{(X, \cal F, \Delta), \epsilon}$ is big.
\end{proposition}

\begin{proof}
Fix 
$\epsilon_0 :=\min\{\min_{j \in J}\frac j3, E(J)\}$, 
cf. Lemma~\ref{cor_bound2} for the definition of $E(J)$. 
Clearly, $\epsilon_0 < E(J)$. 

Let us now fix $(X, \cal F, \Delta) \in \mathfrak D_J$.
As $J$ is finite and $1 \not \in J$, there exists $\delta=\delta(J):=1-\max J$ such that $J \subset [0, 1-\delta]$ and $(X, \cal F, \Delta)$ is $(\epsilon, \delta)$-\alc for all $\epsilon\geq 0$, cf. Lemma~\ref{lem_easy_ed_crit}.
If $K_{(X, \cal F, \Delta), \epsilon_0}$ is pseudo-effective there is nothing to show; 
hence, we may assume that 
$K_{(X, \cal F, \Delta), \epsilon_0}$ 
is not pseudo-effective.

Let 
$\rho\colon X \rightarrow Z$ 
be a run of the 
$K_{(X, \cal F, \Delta), \epsilon_0}$-MMP
which exists and terminates by Theorem~\ref{thm_adj_mmp}.
We set 
$\Xi := \rho_{t\ast}\Delta$ 
and 
$\cal L := \rho_{t\ast}\cal F$.
By Theorem~\ref{thm_adj_mmp}, as
$K_{(X, \cal F, \Delta), \epsilon_0}$ 
is not pseudo-effective, 
$Z$ is endowed with a Mori fibre space structure with respect to 
$K_{(Z, \cal L, \Xi), \epsilon_0}$, 
i.e., there exists a contraction 
$\psi \colon Z \to B$
with 
$\dim Z > \dim B$, 
$\rho(Z/B)=1$ 
and 
$-K_{(Z, \cal L, \Xi), \epsilon_0}$ 
is 
$\psi$-ample.
Moreover, the following properties hold:
\begin{enumerate}
\item[(i)] $Z$ has $\eta$-lc singularities for 
$0<\eta=
\frac{\epsilon_0 \delta}{1+\epsilon_0}$, 
see Corollary~\ref{cor_bounded_sing}.
Here we need that $1 \not \in J$ to conclude that $\eta >0$;

\item[(ii)] 
$K_{Z}+\Xi$ is not pseudo-effective: 
it is antiample over $B$;

\item[(iii)] $K_{(Z, \cal L, \Xi), \epsilon_0}$ is antiample over $B$. 
Hence, it is not pseudo-effective on $Z$.
The same holds for $K_{\cal L} + \epsilon_0 K_{Z}$;

\item[(iv)] $K_{(X, \cal F, \Delta), \epsilon_0} = \rho^\ast K_{(Z, \cal L, \Xi), \epsilon_0}+E$ where $E \geq 0$; and,

\item[(v)] $(\cal L, \Xi_{t, n-inv})$ is log canonical.
\end{enumerate}

All of these claims are direct consequences of the negativity lemma and Theorem~\ref{thm_adj_mmp}, Corollary~\ref{cor_bounded_sing}.
We can also show that the geometry of $Z$ is rather restrictive.

\medskip

{\bf Claim 1}. 
{\it 
$Z$ 
is a Fano surface and 
$\rho(Z)=1$, 
i.e., 
$\dim B=0$.
}
\begin{proof}[Proof of Claim 1]
If $\dim B >0$, then $B$ is a curve.
Let 
$F$ 
be a general fibre of 
$\psi$: $F$ 
is rational since 
$\psi$
is a Mori fibre space.
Then, 
$F$ 
is not $\mathcal L$-invariant for the foliation as 
$(K_{\mathcal{L}}+\Xi)\cdot F >0$ 
by the bigness of $K_{\mathcal{L}}+\Xi$;
thus, 
$(K_{\cal L}+\Xi)\cdot F \geq j_0$, 
where 
$j_0 = \min J$,
and 
$(K_{Z}+\Xi)\cdot F \geq -2$. 
Since 
$\epsilon_0 < \frac{j_0}{2}$ 
by definition, then 
$K_{(Z, \cal L, \Xi), \epsilon_0}\cdot F \geq 0$.  
On the other hand, by item (iii) 
$K_{(Z, \cal L, \Xi), \epsilon_0}\cdot F <0$ 
which leads to the sought contradiction.
\end{proof}

Let 
$\mathfrak D_{J, \epsilon _0, Fano}$ 
be the set of foliated triples 
$(Z, \cal L, \Xi)$ 
that appear as final outcomes (i.e., Mori fibre spaces) in a run of the 
$K_{(X, \cal F, \Delta), \epsilon_0}$-MMP 
for 
$(X, \cal F, \Delta) \in \mathfrak D_J$ 
with 
$K_{(X, \cal F, \Delta), \epsilon_0}$ 
not pseudo-effective.
Claim 1 readily implies the following conclusion for 
$\mathfrak D_{J, \epsilon _0, Fano}$.

\medskip

{\bf Claim 2}.
{\it 
$\mathfrak D_{J,\epsilon _0, Fano}$ forms a bounded family.
}

\begin{proof}[Proof of Claim 2]
For any triple 
$(Z, \cal L, \Xi)\in \mathcal D_{J, Fano}$, 
(i) above implies that 
$Z$ is $\eta$-lc 
for some fixed 
$\eta>0$.
By \cite[Theorem~6.9]{Alexeev}, 
$\eta$-lc 
Fano surfaces form a bounded family.
Thus, for any triple 
$(Z, \cal L, \Xi)\in \mathcal D_{J, Fano}$
there exists 
$t = t(\eta) \in \mathbb N_{>0}$ 
such that 
$-tK_{\cal Z}$ 
is very ample.
Furthermore, (ii)-(iii) imply that 
\begin{align}
\label{eqn:fano.ineq}
0 < \Xi \cdot (-K_Z) \leq -K_{Z}^2, 
\qquad 
-\Xi \leq K_{\cal L} \leq -\frac{1}{\epsilon_0}K_{Z}.
\end{align}
Hence, 
$(Z, \Xi)$ 
belong to a bounded family, since 
$\Xi \in J$ 
(and $J$ is finite) and 
$\deg_{-tK_{\cal Z}} \Xi \leq -tK_{\cal Z}^2$.
Moreover, thanks to the fact that 
$K_{\cal L}$ 
is Weil and by~\eqref{eqn:fano.ineq}, then the triples 
$(Z, \Xi, \mathcal O_{\cal Z}(K_{\cal L}))$ 
belong to a bounded family. 
Possibly stratifying $T$ into a disjoint union of locally closed subsets (which does not alter boundedness), we may assume that items (1) and (2) of Lemma~\ref{lem:bounded.fols}
 are satisfied.  By \cite[Th\'eor\`eme 12.2.1 (v)]{EGAIV} and perhaps stratifying further we may assume that item (3) holds as well.
We may then apply 
Lemma~\ref{lem:bounded.fols} to conclude.
\end{proof}

Given 
$(Z, \cal L, \Xi) \in \mathcal D_{J, \epsilon_0, Fano}$, 
by Lemma~\ref{lem_resolution_logsmooth_near_lc}, we may find a partial resolution 
$g\colon Y \rightarrow Z$ together with a morphism $\pi\colon X \rightarrow Y$
such that 
\begin{enumerate}[label = (\alph*)]

\item 
$(Y, \Gamma)$ 
is log smooth near all strictly log canonical singularities of
$\cal G$, 
where 
$\Gamma \coloneqq \pi_{\ast}\Delta$ 
and 
$\cal G \coloneqq \pi_{\ast}\cal F$;

\item $g$ only extracts divisors $E$ with $a(E, Z, \Xi) \leq 0$ and $a(E, \cal L, \Xi) = -\iota(E)$.
In particular there exists a $g$-exceptional divisor $F \geq 0$
so that $K_{(Y, \cal G, \Gamma), \epsilon_0} +F = g^\ast(K_{(Z, \cal L, \Xi), \epsilon_0})$

\item $(\cal G, \Gamma)$ is log canonical (this is a direct consequence of item (b)) and $K_{\cal G}+\Gamma$ is big;

\item 
$(Y, \Gamma)$ is log smooth at all strictly log canonical singularities of $\cal G$; and,

\item 
if 
$(X, \cal F, \Delta) \in \mathcal D_J$, 
then 
$Z$ 
is a Mori fibre space obtained from a run of a 
$K_{(X, \cal F, \Delta), \epsilon_0}$-MMP
and there exists a contraction 
$\pi\colon X \rightarrow Y$.
\end{enumerate}

Furthermore, the collection $\mathcal D_{J, \epsilon_0, Fano, res}$ of all triples $(Y, \cal G, \Gamma)$ as above is bounded, by Lemma~\ref{lem_resolution_logsmooth_near_lc}.
Thus, there exists a projective family of foliated triples
\begin{align*}
f \colon 
(\mathcal Y, \mathcal G_{\mathcal Y}, \Gamma_{\mathcal Y}) 
\to T
\end{align*}
over a base of finite type 
$T$ 
such that 
$\mathcal G_{\mathcal Y}$ 
is tangent to 
$f$ 
and for any triple 
$(Y, \cal G, \Gamma) \in \mathcal D_{J, \epsilon_0, Fano, res}$
there exists 
$t \in T$ 
and an isomorphism 
$\psi \colon Y \to \mathcal Y_t$ 
with 
$\psi_\ast \Gamma=\Gamma_{t}$ 
and 
$\psi_\ast \mathcal G=\mathcal G_{\mathcal Y, t}$.

The conclusion of the proof then follows from the following claim.

\medskip

{\bf Claim 3}.
{\it
Taking 
$\epsilon = \epsilon_0$, 
we can apply Proposition~\ref{prop_bound_for_nice_family} to 
$\mathfrak D=\mathcal D_J$ 
and to the family
$f \colon 
(\mathcal Y, \mathcal G_{\mathcal Y}, \Gamma_{\mathcal Y}) 
\to T$.
}

\medskip

At this point, we can conclude by defining
$\tau_0$ 
to be the number that Proposition~\ref{prop_bound_for_nice_family} produces in the set-up of Claim 3.
\begin{proof}[Proof of Claim 3]
It suffices to show that all the hypotheses~(1)-(4) of Proposition~\ref{prop_bound_for_nice_family} are satisfied.

Item (1) is (e) above.
Item (3) follows from (a) and (c) above.
Item (4) is (d) above.
Thus, we are left to show that item (2) holds.

Fix 
$(X, \mathcal F, \Delta) \in D_J$ 
and let 
$(Y, \cal G, \Gamma) \in \mathcal D_{J, \epsilon_0, Fano, res}$ 
be a triple obtained as the resolution of a Mori fibre space outcome of the 
$K_{(X, \mathcal F, \Delta), \epsilon_0}$-MMP.
We adopt the same notation as in the previous part of the proof.

Suppose first that every component of 
$\text{exc}(g)$ 
is $\cal G$-invariant.  
In this case, by (b) above 
$(K_{Y} +\Gamma)+F= g^\ast(K_{Z}+\Xi)$ 
and 
$(K_{\cal G} +\Gamma)= g^\ast(K_{\cal L}+\Xi)$, where 
$F \geq 0$.
In particular,
$K_{(X, \cal F, \Delta), \epsilon_0} = 
\pi^\ast K_{(Y, \cal G, \Gamma), \epsilon_0} + \tilde{E}$
where 
$\tilde{E} \geq 0$.  

So suppose that some component of $\text{exc}(g)$ is not $\cal G$-invariant.  Let $C$ be one such component
and let $P = g(C)$.  Note that $P$ is a strictly log canonical singularity
of $(\cal L, \Xi)$ and so $P$ is not contained in the support of 
$\Xi$.  
By Lemma \ref{lem_uniq_lc_place}, 
$C$ is the unique non invariant $g$-exceptional divisor
mapping to $P$.  
Thus, in a neighborhood of $g^{-1}(P)$ we see that $\Gamma$ is supported on at most one divisor.
Note that in this neighbhorhood the support of $\Gamma$ must be smooth, indeed, we know that $(\mathcal G, C)$ is log canonical,
which implies that $C$ is necessarily smooth.
It follows that $(\cal G, \Gamma)$ and $(Y, \Gamma)$ have canonical singularities
and so $(Y, \cal G, \Gamma)$ is $\epsilon$-\acanonical for all $\epsilon>0$, in particular, 
it is $\epsilon_0$-\acanonical.
\end{proof}
\end{proof}

\begin{proposition}
\label{prop_shrink_boundary}
Let $I \subset [0, 1]$ be a DCC set.  
Then there exists a finite subset 
$J \subset (I \cup \mathcal S) \setminus \{1\}$, 
with the following property:

Let $(X, \cal F, \Delta)$ be a projective foliated log smooth triple such that 
$X$ is a surface, 
$\mathcal F$ is rank one, 
$K_{\cal F}+\Delta$ is big and $\Delta \in I$.
Then there exists $\Delta' \leq \Delta$ with $\Delta' \in J$ such that $K_{\cal F}+\Delta'$ is big.
\end{proposition}

\begin{proof}
If $\Delta=0$, then there is nothing to prove, since $K_\mathcal{F}$ is big in its own right.
Hence, we may assume that $\Delta \neq 0$.
If $K_{\cal F}$ is pseudo-effective then the result is straightforwardly true.
In fact, since $I$ is a DCC set then $I \setminus \{0\}$ admits a minimum, call it $i_1$. It then suffices to take $J=\{i_1\}$: indeed, writing $\Delta = \sum_k \mu_{D_k} D_k$ the decomposition into prime components, then for $\Delta':= \sum_k i_1 D_k$, $0< \Delta' \leq \Delta$, $\Delta' \in J$ and $K_\mathcal{F}+\Delta'$ is big.
Hence, we may assume that $K_{\cal F}$ is not pseudo-effective.

We first claim that there exists a finite set 
$J \subset (I \cup \mathcal S) \setminus \{ 1 \}$ 
such that if 
$(C, \Theta)$ 
is a log canonical pair on a smooth curve with 
$K_C+\Theta$ 
ample and 
$\Theta \in I$,
then there exists 
$\Theta' \leq \Theta$
with 
$\Theta' \in J$ 
so that 
$K_C+\Theta'$ is ample.
Indeed, by~\cite[Theorem 1.3]{HMX12} there exists an $m$, depending only on $I$, so that
the map induced by 
$|m(K_C+\Theta)|= 
|\lfloor m(K_C+\Theta) \rfloor|$ 
is birational.
For all 
$1 \leq k \leq m$, 
we set 
\begin{align*}
J_k \coloneqq 
\begin{cases}
\min \{(I \cup \mathcal S) \cap [\frac{k-1}{m}, \frac{k}{m}]\} & 
\text{if $(I \cup \mathcal S) \cap [\frac{k-1}{m}, \frac{k}{m}] \neq \emptyset$},\\
1 & \text{otherwise}.
\end{cases}
\end{align*}
When 
$(I \cup \mathcal S) \cap [\frac{k-1}{m}, \frac{k}{m}] \neq \emptyset$,
then 
$\min \{(I \cup \mathcal S) \cap [\frac{k-1}{m}, \frac{k}{m}]\}$
is well-defined since $I \cup \mathcal S$ satisfies the DCC.  
Then it suffices to set 
$J \coloneqq \{ J_k\ \vert \ 1 \leq k \leq m\}$ 
and observe that 
$J$ 
satisfies all our required properties.
When $I=\{1\}$, we cannot just work with $I$ but we really need to work with the set $(I \cup \mathcal S) \setminus \{1\}$ to guarantee that $J \subset [0, 1)$. 

By \cite{bm16}, $\cal F$ is uniruled;
moreover, since $\cal F$ has canonical singularities, then there exists a morphism $X \rightarrow B$ inducing $\cal F$.
Let $C$ be a general fibre of $X \rightarrow B$ and set $\Theta := \Delta\vert_C$.
Thus, we may find $\Delta' \leq \Delta$ with $\Delta' \in J$ such that $\Delta'\vert_C =\Theta'$.

We claim that $K_{\cal F}+\Delta'$ is pseudo-effective, from which we may conclude.
To see the claim, observe that if 
$K_{\cal F}+\Delta'$ 
is not pseudo-effective then 
$X$ 
is covered by 
$(K_{\cal F}+\Delta')$-negative 
rational curves tangent to the foliation, see Theorem~\ref{thm_recall_surf_mmp}. 
However, by construction 
$K_{\cal F}+\Delta'$ 
is positive on a general rational curve tangent to 
$\cal F$. 
\end{proof}

\begin{corollary}
\label{cor_big_thresh_boundary}
Fix a DCC set $I \subset [0, 1]$.
Let $(X, \cal F, \Delta)$ be a projective foliated log smooth triple such that 
$X$ is a surface, 
$\mathcal F$ is rank one, 
$K_{\cal F}+\Delta_{n-inv}$ is big,  
and $\Delta \in I$.
Then there exists a positive real number 
$\tau=\tau(I)$ 
such that for all 
$0\leq \epsilon <\tau$, 
$K_{(X, \cal F, \Delta), \epsilon}$ is big.
\end{corollary}
\begin{proof}
We may assume without loss of generality that $\Delta = \Delta_{n-inv}$.
By Proposition \ref{prop_shrink_boundary} we may assume without loss of generality that there exists a finite subset 
$J \subset (I \cup \mathcal S) \setminus \{1\}$ 
and 
$\Delta' \leq \Delta$
such that 
$\Delta' \in J$ such that 
$K_{\mathcal F}+\Delta'$
is big.
We may then conclude by Proposition \ref{prop_boundary_in_finite_set}, defining 
$\tau\coloneqq \min(\tau_0(J), E(I))$, 
where 
$\tau_0(J)$ 
is the positive real number produced by Proposition~\ref{prop_boundary_in_finite_set} and 
$E=E(I)$ 
is the positive real number defined in Proposition~\ref{cor_bound2}.
\end{proof}

\begin{corollary}
\label{cor_main_corollary}
Fix a DCC set $I \subset [0, 1]$.
Then there exists a positive real number 
$\tau=\tau(I)$ 
such that for all 
$0< \epsilon<\tau$ 
the following statement holds:

Let $(X, \cal F, \Delta)$ be a $\epsilon$-\alc projective foliated triple such that 
$X$ is a surface, 
$K_{\cal F}+\Delta_{n-inv}$ is big, and 
$\Delta \in I$.
There exists an integer $M=M(\epsilon)$ for which $|MK_{(X, \cal F, \Delta), \epsilon}|$ defines a birational map.
\end{corollary}

\begin{proof}
We define $\tau \coloneqq \tau(I)$, where $\tau(I)$ is the positive real number produced by Corollary~\ref{cor_big_thresh_boundary}.

We fix 
$0<\epsilon<\tau$ 
and a projective foliated triple 
$(X, \mathcal F, \Delta)$ 
satisfying the hypotheses of the statement.
As 
$\epsilon< E(I)$ 
by construction, cf. the proof of Corollary~\ref{cor_big_thresh_boundary}, Proposition~\ref{cor_bound2} implies that
$(\mathcal F, \Delta_{{\rm n-inv}})$ 
is log canonical.  
Let 
$p\colon X' \rightarrow X$ 
be a foliated log resolution of 
$(X, \mathcal F, \Delta)$
and let
$\mathcal F' \coloneqq p^{-1}\mathcal F$,
${\rm exc}(p) \coloneqq E$, 
and 
$\Gamma \coloneqq p_\ast^{-1}\Delta+E$.  
Since
$K_{\mathcal F'}+\Gamma_{{\rm n-inv}} = p^\ast (K_{\mathcal F}+\Delta_{{\rm n-inv}})+F$, 
where 
$F\geq 0$, 
then  
$K_{\mathcal F'}+\Gamma_{{\rm n-inv}}$
is big.  
Moreover, 
$K_{(X', \mathcal F', \Gamma), \epsilon} = 
p^\ast K_{(X, \mathcal F, \Delta), \epsilon}+G$, 
where $G \geq 0$.
Hence, for all 
$m \in \mathbb N$, 
$|mK_{(X', \mathcal F', \Gamma), \epsilon}| = 
|mK_{(X, \mathcal F, \Delta), \epsilon}|$, 
and if, for some 
$m \in \mathbb N_{>0}$, 
$|mK_{(X', \mathcal F', \Gamma), \epsilon}|$ 
defines a birational map, then the same holds for 
$|mK_{(X, \mathcal F, \Delta), \epsilon}|$.  
Therefore we are free to replace
$(X, \mathcal F, \Delta)$ 
by 
$(X', \mathcal F', \Gamma)$,
and, thus, we may freely assume that 
$(X, \mathcal F, \Delta)$ 
is a foliated log smooth triple.

By Proposition \ref{prop_shrink_boundary} we may assume without loss of generality that there exists a finite subset 
$J \subset (I\cup \mathcal S) \setminus \{ 1\}$ 
and 
$\Delta' \leq \Delta$ 
with 
$\Delta' \in J$, 
and 
$K_\mathcal{F} +\Delta'$ is big.
Corollary \ref{cor_big_thresh_boundary} in turn implies that
$K_{(X, \cal F, \Delta'), \epsilon}$ is big.
We run a 
$K_{(X, \cal F, \Delta'), \epsilon}$-MMP, 
$\rho\colon X \rightarrow Y$ 
which must terminate with the ample model 
$Y$ 
for 
$K_{(X, \cal F, \Delta'), \epsilon}$, 
see Corollary~\ref{cor_can_model}.
By Corollary \ref{cor_bounded_sing} and the fact that 
$J \subset [0, 1-\delta] \cap I$ 
for some 
$\delta>0$, 
$Y$ 
has 
$\eta$-lc 
singularities for 
$\eta \coloneqq 
\frac{\epsilon\delta}{1+\epsilon} > 0$.

We may then apply Lemma \ref{lem_main_bound} to 
$N \coloneqq 
\frac{1}{\epsilon}\rho_\ast K_{(X, \cal F, \Delta'), \epsilon}$
to conclude, 
since $K_{(X, \cal F, \Delta), \epsilon} \geq K_{(X, \cal F, \Delta'), \epsilon}$.
\end{proof}

\section{Applications}

\subsection{Bounding degrees of curves invariant by foliations}

The following is an improvement on a bound
proven in \cite[Theorem~5.4]{PS16}.

\begin{theorem}
\label{thm_degree_bound}
Let $\tau=\tau(\emptyset)>0$ be the real constant defined within Corollary~\ref{cor_big_thresh_boundary}. 
Then for all positive rational numbers $0<\epsilon<\tau$ there exists positive integer $C = C(\epsilon)$ the following statement holds:

Let $(X, \cal F)$ be a projective foliated pair such that 
\begin{enumerate}
\item 
$X$ is a surface,
\item 
$K_{\mathcal F}$ is big,
\item 
$(X, \mathcal F)$ is $\epsilon$-\acanonical,
\item 
$\cal F$ admits a meromorphic first integral, and 
\item 
\label{item:mer.1st.int.genus}
the closure of a general leaf, $L$, has geometric genus $g$.
\end{enumerate}
Then for any nef divisor $H$ on $X$, 
\begin{align*}
H\cdot L \leq gC( H\cdot (K_{\cal F}+\epsilon K_X)).
\end{align*}
\end{theorem}

\begin{proof}
Let $p\colon X' \rightarrow X$ be a foliated log resolution of $\mathcal F$. 
We define $\mathcal F' = p^{-1}\mathcal F$ and $L' = p_*^{-1}L$.
As $\cal F$ possesses a meromorphic first integral and by \eqref{item:mer.1st.int.genus}, we can assume that $\mathcal F'$ is given by a fibration in curves of genus $g$.

Since $(X, \mathcal F)$ is $\epsilon$-\acanonical 
$K_{\mathcal F'}+\epsilon K_{X'} = p^*(K_{\mathcal F}+\epsilon K_X)+E$
where $E \geq 0$ and is $p$-exceptional, and for all $m \in \mathbb N$, 
\begin{align*}
H^0(X', m(K_{\cal F'} +\epsilon K_{X'}))=
H^0(X, m(K_{\cal F} +\epsilon K_{X})).
\end{align*}
Moreover, by Corollary \ref{cor_main_corollary}, there exists a positive integer $M=M(\epsilon)$, that is, $M$ independent of $X'$ and $\cal F'$, such that $|M(K_{\cal F'}+\epsilon K_{X'})|$
defines a birational map.  
Thus, $h^0(X', lM(K_{\cal F'}+\epsilon K_{X'})) \geq {l+2 \choose 2}$.

As 
$m(K_{\cal F'}+\epsilon K_{X'})\vert_{L'} ) = m(1+\epsilon)K_{L'}$, 
then for all $m >1$, 
$h^0(L', m(K_{\cal F'}+\epsilon K_{X'})\vert_{L'}) \leq m(1+\epsilon)(2g-2)-g+1$.
We therefore have inequalities
\begin{align*}
h^0(X', l(K_{\cal F'}+\epsilon K_{X'})-L') \geq &
h^0(X', l(K_{\cal F'}+\epsilon K_{X'})) \\
- & h^0(L', l(K_{\cal F'}+\epsilon K_{X'}))\vert_{L'})
>1,
\end{align*}
where the latter inequality holds for $l \geq 4gM^2$.
Therefore, there exists $0 \leq D \sim l(K_{\cal F'}+\epsilon K_{X'})-L'$.
Since $H$ is nef, then $0 \leq D\cdot p^*H$. 
Hence, taking $C(\epsilon) \coloneqq 4M^2$, we obtain the desired result.
\end{proof}

\subsection{Lower bound on adjoint volumes}

\begin{theorem}
\label{thm_vol_bound}
Fix a DCC set $I \subset [0,1]$.
Let $\tau=\tau(I)>0$ be the real constant defined within Corollary~\ref{cor_big_thresh_boundary}. 
Then for all $0<\epsilon<\tau$ there exists $0< v(\epsilon)$ such that the following statement holds:

If $(X, \cal F, \Delta)$ is an $\epsilon$-\alc foliated projective triple where $\Delta \in I$, 
$X$ is a surface and $K_{\cal F}+\Delta$ is big then 
\begin{align*}
\mathrm{vol}(K_{(X, \cal F, \Delta), \epsilon}) \geq v(\epsilon).
\end{align*}
\end{theorem}
\begin{proof}
This is a direct consequence of Corollary \ref{cor_main_corollary}.
\end{proof}

\subsection{Upper bound on automorphism group of foliations}

\begin{theorem}
\label{thm_automorphism_bound}
Let $\tau=\tau(\cal S)>0$ be the real constant defined within Corollary~\ref{cor_big_thresh_boundary}. 
Then, for all $0<\epsilon<\tau$ there exists $C = C(\epsilon)$
such that the following holds:

Let $(X, \cal F)$ is a projective foliated pair such that 
\begin{enumerate}
\item 
$X$ is a surface,
\item 
$K_{\mathcal F}$ is big,
\item 
$(X, \mathcal F)$ is $\epsilon$-\acanonical.
\end{enumerate}
Then
\begin{align*}
\# \bir(X, \cal F) \leq C\cdot \mathrm{vol}(K_{(X, \cal F), \epsilon}).
\end{align*}
\end{theorem}
\begin{proof}
By \cite{PS02} we know that $\#\bir(X, \cal F)<+\infty$. 
Possibly replacing $(X, \cal F)$ by a higher model,
we may freely assume that $(X, \cal F)$ is log smooth, 
\begin{align*}
\bir(X, \cal F) = \autom(X, \cal F) = G,
\end{align*}
and that, if 
$Y = X/G$ and $\cal G = \cal F/G$, then $(Y, \cal G, \Delta)$ 
is log smooth, where 
$\Delta = \displaystyle \sum_{D  \text{ prime}}  \frac{r_D-1}{r_D}D$ 
and $r_D$ is the ramification index of $q$ over the prime divisor $D \subset Y$.

By Riemann-Hurwitz and foliated Riemann-Hurwitz, \cite[Lemma 3.4]{Druel18}, then
$K_{(X, \cal F), t} = q^*K_{(Y, \cal G, \Delta), t}$ 
for all 
$t\geq 0$.
Thus, 
$\#G \leq \frac{\volume(K_{(X, \cal F), t})}{\volume(K_{(Y, \cal G, \Delta), t})}$.

Set 
$\tau \coloneqq \tau(\cal S)$ 
and 
$C(\epsilon) \coloneqq \frac{1}{v(\epsilon)}$
for 
$\epsilon<\tau$, 
where $\tau(\cal S)$ and $v(\epsilon)$ are as in Theorem~\ref{thm_vol_bound}.
This gives our desired bound.
\end{proof}

\section{Boundedness of ample models}

In this section we fix $\tau$ to be $\tau(\emptyset)$, the real number whose existence has been shown in Corollary~\ref{cor_big_thresh_boundary}.

\begin{theorem}
\label{thm_bdd_ample_models}
Fix positive real numbers $C$ and $\epsilon$, with $\epsilon < \tau$.

The set $\mathcal{M}_{2, \epsilon, C}$ of foliated pairs $(X, \mathcal F)$ such that 
\begin{enumerate}
\item 
$X$ is a projective klt surface,

\item
$\cal F$ is a rank one foliation on $X$ with $K_\mathcal{F}$ big, 

\item 
$(X, \mathcal F)$ is an $\epsilon$-\acanonical foliated pair, 

\item 
$K_{\cal F}+\epsilon K_X$  is ample, and

\item 
$\mathrm{vol}(X, K_{\cal F}+\epsilon K_X) \leq C$
\end{enumerate}
forms a bounded family.
\end{theorem}

\begin{proof}
We shall divide the proof into several steps. 

\medskip

{\bf Step 1}. 
{\it Effective birational boundedness.}

By Corollary~\ref{cor_main_corollary}, we know that if $(X, \mathcal F)$ is one of the pairs in $\mathcal{M}_{2, \epsilon, C}$, then there exists an integer $M=M(\epsilon)$ such that the morphism $|M(K_{(X, \mathcal F), \epsilon})|$ is birational onto the image.

Let $Y$ be the Zariski closure of the image of $X$ under the induced map.
As $\mathrm{vol}(X, K_{(X, \mathcal F), \epsilon}) \leq C$ then $Y$ belongs to a bounded family.
Let $H \in |\mathcal O_Y(1)|$ be a general member.

\medskip

{\bf Step 2}.
{\it Normalization and boundedness.}

Let $\nu \colon Y^\nu \to Y$ be the normalization of $Y$. 
And let $H_\nu=\nu^\ast H$.
Then, also $Y^\nu$ belongs to a bounded family and we can assume that there exists a positive integer $l=l(\epsilon, C)$ such that $lH_\nu$ is very ample by Matsusaka's big theorem for normal surfaces, see~\cite[\S~3]{mats}.

\medskip

{\bf Step 3}. 
{\it Relatively ample model.}

Let $(X^{(1)}, \mathcal F^{(1)})$ be a foliated log resolution of $(X, \mathcal F)$ with morphism $r \colon X^{(1)} \to X$.
The $\epsilon$-\acanonical condition for $(X, \mathcal F)$ implies that 
\begin{align*}
& K_{(X^{(1)}, \mathcal F^{(1)}), \epsilon}=
r^\ast (K_{(X, \mathcal F), \epsilon}) + E,
\quad E \geq 0 \text{ and } \text{$r$-exceptional, and}\\
& H^0(X, mK_{(X, \mathcal F), \epsilon})= 
H^0(X^{(1)}, mK_{(X^{(1)}, \mathcal F^{(1)}), \epsilon}), 
\quad \forall m \in \mathbb N.
\end{align*}

We note that $K_{\mathcal F^{(1)}}$ is pseudo-effective. Indeed, if not then $\mathcal F^{(1)}$ is induced
by a fibration in rational curves $\rho\colon X^{(1)} \rightarrow B$.  If $\Sigma$ is a general
fibre of $\rho$ then $(K_{\mathcal F^{(1)}}+\epsilon K_{X^{(1)}})\cdot \Sigma = -2(1+\epsilon)$, a contradiction 
of the pseudo-effectivity of $K_{\mathcal F^{(1)}}+\epsilon K_{X^{(1)}}$.

By taking a sufficiently high model $X^{(1)}$, we can assume that there exists a morphism $p \colon X^{(1)} \to Y$ which is a resolution of indeterminacies of the rational map $X \dashrightarrow Y$ constructed in Step 1.
As $X^{(1)}$ is smooth, then it automatically factors through the normalization, thus inducing $p_1 \colon X^{(1)} \to Y^\nu$.
We will denote by $\mathcal F_{Y^\nu}$ the strict transform of the foliation on $Y^\nu$.
In view of this, and by construction, cf. Step 1, then 
\begin{align*}
|M(K_{(X^{(1)}, \mathcal F^{(1)}), \epsilon})|= 
p_1^\ast|H_\nu| + F, \quad F \geq 0.
\end{align*}
In particular, for any $t \in \mathbb{R}_{\geq 0}$,
\begin{align*}
\mathrm{vol}(X^{(1)}, K_{(X^{(1)}, \mathcal F^{(1)}), \epsilon} + tp_1^\ast H_\nu)
\leq &
\mathrm{vol}(X^{(1)}, (1+tM)K_{(X^{(1)}, \mathcal F^{(1)}), \epsilon}) \\
\leq &
(1+tM)^2 C.
\end{align*}

Let us run the $K_{(X^{(1)}, \mathcal F^{(1)}), \epsilon}$-MMP relatively over $Y^\nu$, which exists by Corollary~\ref{cor_rel_mmp},
\begin{align*}
\xymatrix{
X^{(1)} \ar[rr]^{s'} \ar[dr]_{p_1}& 
&
X^{(2)} \ar[dl]^{p_2}
\\
&
Y^\nu.
&
}
\end{align*}
Moreover, by Remark~\ref{rmk_rel_ample_model}, we can pass to the ample model (over $Y^\nu$) for the $K_{(X^{(1)}, \mathcal F^{(1)}), \epsilon}$-MMP, that is, we can assume that $X^{(2)}$ satisfies the following conditions:
\begin{enumerate}[label=(\roman*)]
    \item 
$X^{(2)}$ is $\eta$-lc for some $\eta=\eta(\epsilon)>0$;
    \item
$(X^{(2)}, \mathcal F^{(2)})$ is $\epsilon$-\acanonical, where $\mathcal F^{(2)}\coloneqq s'_*\mathcal F^{(1)}$;
    \item
$|MK_{(X^{(2)}, \mathcal F^{(2)}), \epsilon}-p_2^\ast H_\nu| \neq \emptyset$;
    \item 
$K_{(X^{(2)}, \mathcal F^{(2)}), \epsilon}$ is ample over $Y^\nu$;
    \item 
$H^0(X^{(2)}, mK_{(X^{(2)}, \mathcal F^{(2)}), \epsilon})=
H^0(X^{(1)}, mK_{(X^{(1)}, \mathcal F^{(1)}), \epsilon}), 
\ \forall m \in \mathbb N$, 
and for any $t \in \mathbb{R}_{\geq 0}$ 
\begin{align}
\label{eqn.est.vol.X^{(2)}}
& \mathrm{vol}(X^{(2)}, K_{(X^{(2)}, \mathcal F^{(2)}), \epsilon} + tp_2^\ast H_\nu)\\
\nonumber
\leq &
\mathrm{vol}(X^{(2)}, (1+tM)K_{(X^{(2)}, \mathcal F^{(2)}), \epsilon}) 
\leq
(1+tM)^2 C.
\end{align}
\end{enumerate}
Item (iv) and  the Cone Theorem for surface foliations,~\cite[Theorem~6.3]{Spicer17},
imply that $K_{(X^{(2)}, \mathcal F^{(2)}), \epsilon} + (1+\lceil \epsilon \rceil)7p_2^\ast H_\nu$ is ample on $X^{(2)}$. 
Item (v) implies that 
\begin{align*}
\mathrm{vol}(X^{(2)}, K_{(X^{(2)}, \mathcal F^{(2)}), \epsilon} + (1+\lceil \epsilon \rceil)7p_2^\ast H_\nu) \leq
(1+(1+\lceil \epsilon \rceil)7M)^2 C.
\end{align*}

{\bf Step 4}. 
{\it Boundedness of intersection numbers.}

As 
$K_{(X^{(2)}, \mathcal F^{(2)}), \epsilon} + (1+\lceil \epsilon \rceil)7p_2^\ast H_\nu$ 
is ample, then the same holds for 
$K_{(X^{(2)}, \mathcal F^{(2)}), \epsilon} + ((1+\lceil \epsilon \rceil)7+s)p_2^\ast H_\nu$, 
for any $s \in \mathbb{R}_{\geq 0}$.
Denoting
\begin{align*}
f(s)=\mathrm{vol}(X^{(2)}, K_{(X^{(2)}, \mathcal F^{(2)}), \epsilon} + ((1+\lceil \epsilon \rceil)7+s)p_2^\ast H_\nu),
\end{align*}
then for all $0\leq s<1$,
\begin{align}
\label{eqn.est.vol.X^{(2)}.2}
& (1+(2+\lceil \epsilon \rceil)7M)^2 C \geq f(s) - f(0) = \int_0^s f'(x)dx =\\
\nonumber
= & \int_0^s
\frac{d}{dt} \vert_{t=x} 
\left[
\mathrm{vol}(X^{(2)}, K_{(X^{(2)}, \mathcal F^{(2)}), \epsilon} + ((1+\lceil \epsilon \rceil)7+t)p_2^\ast H_\nu) 
\right]
dx
\\
\nonumber
= & \int_0^s
\frac{d}{dt}\vert_{t=x} 
\left[
(K_{(X^{(2)}, \mathcal F^{(2)}), \epsilon} + ((1+\lceil \epsilon \rceil)7+t)p_2^\ast H_\nu)^2 
\right]
dx
\\
\nonumber
= & 
\int_0^s
\left[
2 K_{(X^{(2)}, \mathcal F^{(2)}), \epsilon} \cdot p_2^\ast H_\nu+ 2((1+\lceil \epsilon \rceil)7+x)H_\nu^2
\right]dx
\\
\nonumber
= & \ 
2((K_{\mathcal F^{(2)}} +\epsilon K_{X^{(2)}})\cdot p_2^\ast H_\nu) s + 2H_\nu^2((1+\lceil \epsilon \rceil)7) s + H_\nu^2 s^2.
\end{align}
As $Y^\nu$ is bounded, $lH_\nu$ is very ample and $H_\nu^2 \leq C$, cf.~Step~2,
$H_\nu^2$ can only take finitely many values in the positive integers.
Likewise, by the push-pull formula $K_{X^{(2)}} \cdot p_2^\ast H_\nu = K_{Y^\nu} \cdot H_\nu
$ 
can only take finitely many integral values.
The boundedness of $p_2^\ast H_\nu^2$ and of $K_{X^{(2)}} \cdot p_2^\ast H_\nu$, 
together with the last line of~\eqref{eqn.est.vol.X^{(2)}.2}, implies that also $K_{\mathcal{F}^{(2)}} \cdot p_2^\ast H_\nu$ can only take finitely many values in the positive integers:
the positivity of  $K_{\mathcal{F}^{(2)}} \cdot p_2^\ast H_\nu$ follows from the pseudo-effectivity of $K_{\mathcal F^{(2)}}$, and since $H_\nu$ is big and nef.
Hence, there exists a finite set $L_{t, \epsilon}=L_{t, \epsilon}(t, \epsilon, C) \subset \mathbb N_{>0}+\epsilon \mathbb N_{>0} +t\mathbb N_{>0}$ such that 
\begin{align*}
(K_{(X^{(2)}, \mathcal F^{(2)}), \epsilon} + t p_2^\ast H_\nu) \cdot p_2^\ast H_{\nu} \in L_{t, \epsilon}.
\end{align*}
\medskip
{\bf Step 5}. 
{\it A new model.}

Starting with $X^{(2)}$, we now run the $K_{X^{(2)}}$-MMP over $Y^\nu$ and then pass to the canonical model over $Y^\nu$, which exists by (i) in Step 3,
\begin{align*}
\xymatrix{
X^{(1)} \ar[r]^{s'} \ar[dr]_{p_1}& 
X^{(2)} \ar[d]^{p_2} \ar[r]^{s''} &
X^{(3)} \ar[dl]^{p_3}&
\\
&
Y^\nu
&
}
\end{align*}
Let $\mathcal F^{(3)} \coloneqq s''_*\mathcal F^{(2)}$.
By Step 4 and the push-pull formula, there exists a positive real number $C'_{t, \epsilon}=C'(t, \epsilon, C)$ such that
\begin{align*}
0 < (K_{(X^{(3)}, \mathcal F^{(3)}), \epsilon} + t p_3^\ast H_\nu) \cdot p_3^\ast H_{\nu} \leq C'_{t, \epsilon}.
\end{align*}
Hence, there exists a positive real number $C''_{t, \epsilon}=C''_{t, \epsilon}(t, \epsilon, C)$ such that
\begin{align*}
0 < (K_{(X^{(3)}, \mathcal F^{(3)}), \epsilon} + t p_3^\ast H_\nu)^2 \leq 
\frac{((K_{(X^{(3)}, \mathcal F^{(3)}), \epsilon} + t p_3^\ast H_\nu) \cdot p_3^\ast H_{\nu})^2}{(p_3^\ast H_{\nu})^2} \leq 
C''_{t, \epsilon}.
\end{align*}
where the second inequality holds by the Hodge index theorem, i.e., $(A^2)(B^2) \leq (A\cdot B)^2$
where $A$ is nef divisor.

\medskip

{\bf Step 6}. 
{\it Boundedness of $X^{(3)}$.}

Let us observe that $X^{(3)}$ is $\eta$-lc, by construction.  
We now show that  $X^{(3)}$ belongs to a bounded family.
By the classical version of the Cone Theorem,~\cite[Theorem~3.7]{KM98}, $K_{X^{(3)}}+7p_3^\ast H_\nu$ is ample on $X^{(3)}$;
furthermore, 
\begin{align}
\label{eqn:vol.estim.1}
&\mathrm{vol}(X^{(3)}, K_{X^{(3)}}+7p_3^\ast H_\nu ) 
\\
\nonumber
\leq & 
\frac{1}{\epsilon^2}
\mathrm{vol}(X^{(3)}, K_{(X^{(3)}, \mathcal F^{(3)}), \epsilon}+7\epsilon p_3^\ast H_\nu)
\hspace{.5cm} 
\left[ \text{$K_{\mathcal F^{(3)}}$ is pseudo-effective}\right]
\\
\nonumber
\leq & 
\frac{C''_{7\epsilon, \epsilon}}{\epsilon^2} .
\end{align}
By Bertini's theorem, taking a general (irreducible) divisor $H' \in |7lp_3^\ast H_\nu|$, where $l=l(\epsilon, C)$ is the natural number defined in Step 2, we can ensure that $(X^{(3)}, \frac{H'}{l})$ is $\eta'$-lc, for $\eta':=\min(\eta, \frac 1l)$;
thus,~\cite[Theorem~1.3]{Fil} and~\eqref{eqn:vol.estim.1} together imply that 
$
\mathrm{vol}(X^{(3)}, K_{X^{(3)}}+7p_3^\ast H_\nu)=\mathrm{vol}(X^{(3)}, K_{X^{(3)}}+\frac{H'}{l})
$
belongs to a finite set;
finally, by~\cite[Theorem~6]{MST} $(X^{(3)}, H')$ is bounded.
Thus, for any $s \in \mathbb R_{>0}$,
\begin{align}
\label{eqn:vol.mod.3}
& \mathrm{vol}(X^{(3)}, K_{(X^{(3)}, \mathcal F^{(3)}), \epsilon} + 6 p_3^\ast H_\nu +s (K_{X^{(3)}}+7p_3^\ast H_\nu))
\\
\nonumber
\leq & \
\mathrm{vol}(X^{(3)}, (1+s)K_{(X^{(3)}, \mathcal F^{(3)}), \epsilon} + (6+7s) p_3^\ast H_\nu)
\hspace{.5cm} \left[ \text{$K_{\mathcal F^{(3)}}$ is pseudo-effective}\right]
\\ 
\nonumber
\leq &
(1+s)^2
\mathrm{vol}(X^{(3)}, K_{(X^{(3)}, \mathcal F^{(3)}), \epsilon} + \frac{6+7s}{1+s} p_3^\ast H_\nu)
\\ 
\nonumber
\leq &
(1+s)^2 C''_{\frac{6+7s}{1+s}, \epsilon}
\end{align}

{\bf Step 7}. 
{\it Boundedness of $\mathcal F^{(3)}$.}

By Step 3, the divisor $K_{(X^{(2)}, \mathcal F^{(2)}), \epsilon} + (1+\lceil \epsilon \rceil)7p_2^\ast H$ is ample on $X^{(2)}$, hence its pushforward $K_{(X^{(3)}, \mathcal F^{(3)}), \epsilon} + (1+\lceil \epsilon \rceil)7p_3^\ast H_\nu$ is big and nef on $X^{(3)}$.
\\
As in Step 4, we compute the derivative 
\begin{align*}
\frac{d}{ds} &
\mathrm{vol}(X^{(3)}, K_{(X^{(3)}, \mathcal F^{(3)}), \epsilon} + (1+\lceil \epsilon \rceil)7 p_3^\ast H_\nu +s (K_{X^{(3)}}+7p_3^\ast H_\nu))
\\
=\frac{d}{ds} &
(K_{(X^{(3)}, \mathcal F^{(3)}), \epsilon} + (1+\lceil \epsilon \rceil)7p_3^\ast H_\nu +s (K_{X^{(3)}}+7p_3^\ast H_\nu))^2
\\
= \ \ \ \ & 2 (K_{(X^{(3)}, \mathcal F^{(3)}), \epsilon} +(1+\lceil \epsilon \rceil)7p_3^\ast H_\nu) \cdot (K_{X^{(3)}}+7p_3^\ast H_\nu)\\
& + 2s(K_{X^{(3)}}+7p_3^\ast H_\nu)^2.
\end{align*}
By boundedness and Step 3-4, there exists a positive real number $D=D(\epsilon, C)$ such that $K_{X^{(3)}}^2, K_{X^{(3)}} \cdot p_3^\ast H_\nu, K_{\mathcal F^{(3)}} \cdot p_3^\ast H_\nu$ all belong to the interval $[-D, D]$.
Repeating the same argument  about the derivative of the volume, as in Step 4, then we can show that there exists a positive real number 
$D'=D'(\epsilon, C)$ 
such that 
\begin{align*}
0< K_{\mathcal F^{(3)}} \cdot (K_{X^{(3)}} +7 p_3^\ast H_\nu) \leq D'.
\end{align*}
Since for any $\lambda \in \mathbb N_{>0}$,
\begin{align*}
& 2K_{X^{(3)}}+K_{\mathcal F^{(3)}}+((2+\lceil \epsilon \rceil)7+\lambda)p_3^\ast H_\nu\\
= &
K_{X^{(3)}} + \underbrace{K_{(X^{(3)}, \mathcal F^{(3)}), \epsilon} + (1+\lceil \epsilon \rceil)7p_3^\ast H_\nu}_{\text{big and nef}} + \underbrace{(1-\epsilon)K_{X^{(3)}} + (7+\lambda)p_3^\ast H_\nu}_{\text{ample}},
\end{align*}
Kawamata-Viehweg vanishing implies that for $i=1,2$, 
\begin{align*}
h^i(X^{(3)}, 2K_{X^{(3)}}+K_{\mathcal F^{(3)}}+((2+\lceil \epsilon \rceil)7+\lambda)p_3^\ast H_\nu)=0.
\end{align*}
As $H_\nu$ is Cartier, then $\chi(X^{(3)}, 2K_{X^{(3)}}+K_{\mathcal F^{(3)}}+((2+\lceil \epsilon \rceil)7+\lambda)p_3^\ast H_\nu)$ is a degree $2$ polynomial in $\lambda$, cf.~\cite[Theorem~1.6]{HL}.
Thus, for at least one value of $\lambda \in \{0,1,2\}$,
\begin{align*}
|2K_{X^{(3)}}+K_{\mathcal F^{(3)}}+((2+\lceil \epsilon \rceil)7+\lambda)p_3^\ast H_\nu| \neq \emptyset.
\end{align*}
We fix such value of $\lambda$.  
Let $\Gamma \in |2K_{X^{(3)}}+K_{\mathcal F^{(3)}}+((2+\lceil \epsilon \rceil)7+\lambda)p_3^\ast H_\nu|$, then the argument 
from the start of this step shows that there exists a positive real number $D''=D''(\epsilon, C)$ such that 
\begin{align*}
\Gamma \cdot (K_{X^{(3)}}+7p_3^\ast H_\nu) \leq D''.
\end{align*}
Hence, the couple $(X^{(3)}, \mathrm{Supp} \ \Gamma)$ constructed in this step belongs to a bounded family.
Moreover, as $X^{(3)}$ itself is bounded, then we can conclude that there is a bounded family parametrizing the pairs $(X^{(3)}, \mathcal O_X(\Gamma-2K_{X^{(3)}}-((2+\lceil \epsilon \rceil)7+\lambda)p_3^\ast H_\nu)$ and 
\begin{align*}
O_X(\Gamma-2K_{X^{(3)}}-((2+\lceil \epsilon \rceil)7+\lambda)p_3^\ast H_\nu) \simeq \cal O(K_{\cal F_{X^{(3)}}}).
\end{align*}
The dependence on $\lambda$ here does not constitute an issue, since $\lambda \in \{0, 1, 2\}$; thus, up to working with a larger family, we can assume that all 3 possible values of $\lambda$ are considered.
Hence, we have reconstructed the pair 
$(X^{(3)}, \cal O(K_{\mathcal F_{X^{(3)}}}))$,
where $\mathcal F_{X^{(3)}}$ here is considered as an abstract Weil divisorial sheaf on $X^{(3)}$.
\\
Finally, we want to reconstruct the foliation in a family.
Indeed, the properties we proved so far imply that there exists a projective morphism 
$\mathcal X^{(3)} \rightarrow T$, 
where $T$ is a quasi projective variety
and a Weil divisorial sheaf 
$\widetilde{\mathcal K}$ 
such that for any $(X, \mathcal F)$ in $\mathcal{M}_{2, \epsilon, C}$
there exists $t \in T$ such that 
$(X^{(3)}, \cal O(K_{\mathcal F_{X^{(3)}}})) \cong (\widetilde{\mathcal X}_t, \cal O(\widetilde{\mathcal K}\vert_{\mathcal X_t}))$.
Possibly stratifying $T$ into a disjoint union of locally closed subsets (which does not alter boundedness), we may assume that items (1) and (2) of Lemma~\ref{lem:bounded.fols}
 are satisfied.  By \cite[Th\'eor\`eme 12.2.1 (v)]{EGAIV} and perhaps stratifying further we may assume that item (3) holds as well.
We may then apply 
Lemma~\ref{lem:bounded.fols} and produce a pair given by a normal variety $\mathcal X^{(3)}$ and a rank $1$ foliation $\mathcal F_{\mathcal X^{(3)}}$ on $\mathcal{X}^{(3)}$ and a projective morphism $f_3 \colon \mathcal{X}^{(3)} \to T$ of varieties of finite type such that:
\begin{itemize}
    \item 
$\mathcal F^{(3)}$ is tangent to the fibers of $f_3$; and,
    \item 
for any $(X, \mathcal F) \in \mathcal M_{2, \epsilon, C}$ there exists $t \in T$ and an isomorphism $\psi_t \colon \mathcal X^{(3)}_t \to X^{(3)}$ identifying $\mathcal F^{(3)}$ with $\mathcal F_{\mathcal X^{(3)}} \vert_{\mathcal X^{(3)}_t}$, where $(\mathcal X^{(3)}, \mathcal F^{(3)})$ is the model constructed in Step 5 -- starting from $(X, \mathcal F)$.
\end{itemize}

\medskip

{\bf Step 8}.
{\it 
Resolutions in families:
in this step we show that
then there exists a bounded family $(f_4 \colon \mathcal{X}^{(4)} \to T, \mathcal F_{\mathcal X^{(4)}})$ of foliated pairs such that for any pair $(X, \mathcal F) \in \mathcal M_{2, \epsilon, C}$ there exists $t \in T$ such that
\begin{itemize}
    \item 
$X_t^{(4)}$ is smooth;
    \item 
$(X_t^{(4)}, \mathcal F_t^{(4)})$ is $\epsilon$-\acanonical; and,
    \item 
$(X_t^{(4)}, \mathcal F_t^{(4)})$ is birational to $(X, \mathcal F)$.
\end{itemize}
}

Continuing with the notation from the end of the previous step, up to passing to a 
finite stratification of $T$ into locally closed subsets we can assume that $T$ is smooth. 
By further stratifying the base $T$ into locally closed subsets and taking resolutions in families, we can moreover construct a pair $(f_4 \colon \mathcal{X}^{(4)} \to T, \mathcal F_{\mathcal X^{(4)}})$ such that gives a family of foliations on smooth projective surfaces, in the sense of \cite[Definition~7.1]{PS16}.
Applying \cite[Proposition~7.4]{PS16}, and further stratifying $T$ into locally closed subsets then yields the desired bounded family:
namely, denoting for $t \in T$ the restriction of $\mathcal F_{\mathcal X^{(4)}}$ to $\mxttt$ by $\mfttt$, then the following properties hold:
\begin{enumerate}[label=(\Roman*)]
\item 
for all $t \in T$, $(\mxttt, \mfttt)$ is 
$\epsilon$-\acanonical 
and it is birational to 
$(\mathcal X^{(3)}_t, \mathcal F_t^{(3)})$;

\item 
$\mxtq, g, T$ are all smooth and $T$ is finite type;
\end{enumerate}

\medskip

{\bf Step 9}. 
{\it Properties of the family $(f_4 \colon \mathcal{X}^{(4)} \to T, \mathcal F_{\mathcal X^{(4)}})$}

We define the set $W \subset T$ by
\begin{align*}
W:=\{ t \in T \ \vert & \ (\mathcal X_t^{(4)}, \mathcal F_{\mathcal X_t^{(4)}}) \text{ admits a birational morphism} \\
& (\mathcal X_t^{(4)}, \mathcal F_{\mathcal X_t^{(4)}}) \rightarrow (X, \mathcal F) \text{ to a foliated pair } (X, \mathcal F) \in \mathcal M_{2, \epsilon, C}\}
\end{align*}

Further stratifying $T$ and possibly discarding some components, we can assume that also the following properties hold:

\begin{enumerate}[label=(\Roman*)]
\setcounter{enumi}{2}

\item 
$W \cap \tilde T$ is Zariski dense in each connected component $\tilde T$ of $T$; 
and,

\item 
for all $t \in T$, $K_{(\mxttt, \mfttt), \epsilon}$ is big.
To prove that this property holds, let us recall that for each $t \in W$, $|MK_{(\mxttt, \mfttt), \epsilon}|$ induces a birational map:
indeed, since $(\mxttt, \mfttt)$ is birational equivalent to an $\epsilon$-\acanonical model $(X, \mathcal F) \in \cal M_{2, \epsilon, C}$, then the conclusion follows at once from Corollary~\ref{cor_main_corollary}.
As $W$ is Zariski dense in each connected component of $T$, and $T$ is finite type, by semicontintuity of cohomology groups in family, we can assume that there exists a Zariski dense open subset $W^\circ \subset W$ 
such that for any $t \in W^\circ$, the natural restriction map 
\begin{align*}
H^0(W^\circ, f_\ast \mathcal O_{\mxtq}(MK_{(\mxtq, \mftq), \epsilon}))
\to H^0(\mxttt,\mathcal O_{\mxttt}(MK_{(\mxttt, \mfttt), \epsilon}))\end{align*} 
is surjective.
In particular, this implies that for $t \in W^\circ$ the rational map over $T$ given by 
$f_\ast \mathcal O_{\mxtq}(NK_{(\mxtq, \mftq), \epsilon})$ is birational along $\mxttt$, which then proves the claim, as this is an open condition on $T$.
\end{enumerate}

\medskip

{\bf Step 10}. 
{\it Conclusion. Boundedness of the ample model.}

Let $T_1$ be an irreducible component of $T$ and let $\eta_1 \in T_1$ be its generic point.
Let $\mxtq_{\eta_1}$ be the fibre over $\eta_1$ and $\mxtq_{\bar{\eta}_1}$ the extension of 
scalars to the algebraic closure $\overline{k(T_1)}$ of $k(T_1)$.
We denote by $\mftqe$ the induced foliation on $\mxtqe$. 
By Lemma \ref{lem_easy_ed_crit}, 
$(\mxtqe, \mftqe)$ is $\epsilon$-\acanonical and $K_{(\mxtqe, \mftqe), \epsilon}$ is big.
Hence, by Theorem~\ref{thm_adj_mmp} and Corollary~\ref{cor_can_model}, we can run the 
$K_{(\mxtqe, \mftqe), \epsilon}$-MMP
\begin{align}
\label{eqn:mmp.gen.pt}
\xymatrix{
\mxtqe := \hat{\mathcal X}_0 \ar[r]^{g_0} & 
\hat{\mathcal X}_1 \ar[r]^{g_1} & 
\hat{\mathcal X}_2 \ar[r]^{g_2} & 
\dots \ar[r]^{g_{n-1}} &
\hat{\mathcal X}_n
}
\end{align}
terminating with a projective model 
$(\hat{\mathcal X}_n, \hat{\cal F})$ defined over $\overline{k(T_1)}$ on which  
$K_{(\hat{\mathcal X}_n, \mathcal F_{\hat{\mathcal X}_n}), \epsilon}$ is ample and it has 
$\epsilon$-\acanonical singularities.
As we have a finite number of objects involved in defining all varieties, morphisms, and 
foliations involved in~\eqref{eqn:mmp.gen.pt}, there exists a finite field extension 
$k(T) \subset k'$
over which the $\hat{\mathcal X}_i$, the morphisms $g_i$, 
and the strict transforms $\mathcal F_{\hat{\mathcal X}_i}$ of $\mftqe$ are defined.
Furthermore, as a line bundle $\cal L$ is ample over $k'$ if and only if its extension of scalars to $\overline{k(T_1)}$ is ample over $\overline{k(T_1)}$, then, by construction, $K_{(\hat{\mathcal X}_n, \mathcal F_{\hat{\mathcal X}_n}), \epsilon}$ is defined over $k'$ and already ample on that field.

The discussion in the previous paragraph implies that there exists a dominant \'etale morphism $q \colon U \to T_1$, with $k(U)\simeq k'$ and $q(U) \subset T_1$ open such that
there exist families of foliated surfaces 
$(\widetilde{\mathcal X}_i, \widetilde{\cal F}_i)$ 
projective over $U$ together with morphisms 
$\widetilde{g}_i \colon  \widetilde{\mathcal X}_i \to \widetilde{\mathcal{X}}_{i+1}$ 
and a commutative diagram of morphisms over $U$
\begin{align*}
\xymatrix{
(\widetilde{\mathcal X}_0, \widetilde{\cal F}_0)  
\ar[drr] \ar[r]^{\widetilde{g}_0} & 
(\widetilde{\mathcal X}_1, \widetilde{\cal F}_1) 
\ar[dr] \ar[r]^{\widetilde{g}_1}& 
(\widetilde{\mathcal X}_2, \widetilde{\cal F}_2) \ar[d] \ar[r]^{\widetilde{g}_2}& 
\dots \ar[r]^{\widetilde{g}_{n-1}}&
(\widetilde{\mathcal X}_n, \widetilde{\cal F}_n) \ar[dll] \\
& & U & &
}
\end{align*}

Up to possibly shrinking $U$, we can assume, by construction and the previous observations, that:
\begin{itemize}
\item
for all $i$, the $\widetilde{g}_i$ are all birational morphisms over $U$ and 
$-(K_{(\widetilde{\mathcal X}_i, \widetilde{\cal F}_i), \epsilon})$ is $\widetilde{g}_i$-nef; 
\item
 for all $i$, all fibers of $(\widetilde{\mathcal X}_i, \widetilde{\mathcal F}_i) \to U$ are $\epsilon$-\acanonical, by the negativity lemma applied fibre-wise; 
 \item 
$K_{(\widetilde{\mathcal X}_n, \widetilde{\mathcal F}_n), \epsilon}$ is 
ample when restricted to any fibre of $\widetilde{\mathcal X}_n \to U$.
\end{itemize}
It follows that the fibres of $(\widetilde{\mathcal X}_n, \widetilde{\mathcal F}_n) \to U$ are $\epsilon$-\acanonical models 
and they are birational to the corresponding fibre of 
$(\widetilde{\mathcal X}_0, \widetilde{\mathcal F}_0)$ over the same point on 
$U$. By uniqueness of ample models of $\epsilon$-\acanonical foliated pairs, 
cf. Corollary~\ref{cor_can_model}, they all yield elements of $\mathcal M_{2, \epsilon, C}$.
By Noetherian induction, it suffices to repeat the procedure of this Step a finite number of times, first on each connected component $T_1 \setminus q(U)$ and then on all other connected components to $T$, to show that every element  in $\mathcal M_{2, \epsilon, C}$ must appear as a fibre in the (finitely many) families of surface foliations produced by this process.
\end{proof}

\bibliography{bib.bib}
\bibliographystyle{alpha}
\end{document}